\documentclass[12pt,english,a4paper]{amsart}
\usepackage[T1]{fontenc}
\usepackage{amsmath,amscd}
\usepackage{amssymb}
\usepackage{amsthm}
\usepackage{yfonts,textcomp}
\usepackage{mathrsfs}
\usepackage{MnSymbol}
\usepackage{pictexwd,dcpic}
\usepackage{xfrac}
\usepackage{babel}
\usepackage{enumerate}
\usepackage{hyperref}
\usepackage{enumitem}
\usepackage{stmaryrd}
\hypersetup{
    colorlinks=true,
    linkcolor=blue,
    filecolor=magenta,      
    urlcolor=cyan,
    citecolor=magenta
}
\usepackage{cleveref}
\usepackage{mdwlist} 
\usepackage{todonotes}

\usepackage[margin=1in]{geometry}

\newtheorem{thm}{Theorem}[section]
\newtheorem{cor}[thm]{Corollary}
\newtheorem{lem}[thm]{Lemma}
\newtheorem{prop}[thm]{Proposition}
\newtheorem{claim}[thm]{Claim}
\theoremstyle{definition}
\newtheorem*{defn}{Definition}
\newtheorem{conj}[thm]{Conjecture}

\theoremstyle{remark}
\newtheorem{remark}[thm]{Remark}

\numberwithin{equation}{section}

\newcommand{\cA}{\mathcal{A}}
\newcommand{\cB}{\mathcal{B}}
\newcommand{\cC}{\mathcal{C}}

\newcommand{\cF}{\mathcal{F}}
\newcommand{\cL}{\mathcal{L}}
\newcommand{\cM}{\mathcal{M}}
\newcommand{\cT}{\mathcal{T}}
\newcommand{\cS}{\mathcal{S}}
\newcommand{\cO}{\mathcal{O}}
\newcommand{\cR}{\mathcal{R}}

\newcommand{\QQ}{\mathbb{Q}}
\newcommand{\NN}{\mathbb{N}}
\newcommand{\CC}{\mathbb{C}}
\newcommand{\RR}{\mathbb{R}}

\newcommand{\ZZ}{\mathbb{Z}}
\newcommand{\GG}{\mathbb{G}}

\newcommand{\MM}{\mathbb{M}}
\newcommand{\DD}{\mathbb{D}}

\newcommand{\PR}{\mathrm{PR}}

\newcommand{\val}{\mathrm{val}}

\DeclareMathOperator{\acl}{acl}
\DeclareMathOperator{\dcl}{dcl}
\DeclareMathOperator{\Th}{Th}
\DeclareMathOperator{\RV}{RV}
\DeclareMathOperator{\rv}{rv}

\DeclareMathOperator{\rk}{rk}
\DeclareMathOperator{\rad}{rad}
\DeclareMathOperator{\trdeg}{tr.deg.}

\newcommand\Floris[1]{\leavevmode{\footnotesize\textbf{[Floris: #1]}}}

\title{Hensel minimality, $p$-adic exponentiation and Tate uniformization}
\author{Sebastian Eterovi\'c}
\email{sebastian.eterovic@univie.ac.at}
\address{Kurt G\"odel Research Center, Universit\"at Wien, 1090 Wien, Austria}

\author{Floris Vermeulen}
\email{florisvermeulen.math@gmail.com}
\address{University of Münster, Mathematics Münster, Einsteinstrasse 62, 48149 Münster, Germany}

\date{}

\thanks{S.E.\ was supported by EPSRC fellowship EP/T018461/1 and by the Austrian
Science Fund (FWF) 10.55776/ESP1584024. 
F.V.\ is supported by the Humboldt foundation.}


\begin{document}

\begin{abstract}
    We use Hensel minimality, a non-Archimedean analog of o-minimality, to study several questions around transcendental number theory, unlikely intersections, and differential fields in a non-Archimedean setting.
    In particular, we focus on $p$-adic exponentiation and Tate uniformization on $\CC_p$, which we show live in a Hensel minimal structure on $\CC_p$.
    We start by constructing a large collection of derivations on Hensel minimal fields that respect definable functions, which we then apply to the $p$-adic Schanuel conjecture.
    We also study properties of local definability in analogy to work of Wilkie, and show that $p$-adic Schanuel implies a uniform version of itself.
    For Tate uniformization we show a strong closure property when blurring, and deduce that $\CC_p$ with the \emph{blurred} Tate uniformization is quasiminimal.
    Finally, we prove a result on $p$-adic density of likely intersections for powers of elliptic curves.
\end{abstract}

\maketitle

\section{Introduction}

O-minimality is a framework coming from model theory in which one can carry out analytic arguments with guaranteed tameness. 
We remark that when we say ``analytic arguments'' we specifically mean arguments that are set in either real or complex analysis.
The successful use of o-minimality in important arithmetic geometry problems is now well-established, see for example \cite{pila:andre-oort, tsimerman:andre--ort, pila-shankar-tsimerman} for applications to the Andr\'e--Oort conjecture, \cite{pila-tsimerman:ax-schanuel,mpt,bakker-tsimerman:ax-schanuel-hodge} for results in functions transcendence, and \cite{daw-ren} for results regarding the Zilber--Pink conjecture. 
In part motivated by these achievements, various attempts have been made to come up with a similarly tame framework that supports non-Archimedean analytic arguments, such as P-minimality~\cite{pmin} or C-minimality~\cite{cmin1,cmin2}.
We will specifically focus on the more recent notion of 1-h-minimality~\cite{CHR, CHRV}, which applies more generally than the previous notions.

The purpose of this article is to showcase how 1-h-minimality can be used in analogy to various applications of o-minimality, to exhibit interesting examples of functions definable in 1-h-minimal structures, and to show how to sort some of the technical obstructions that arise.
For this, we will study some non-Archimedean analogs of problems around transcendental number theory, unlikely intersections, and differential fields. 
Most analytic applications of o-minimality start by considering the expansion of the real field $\RR$ by a collection of analytic functions $\mathcal{F}$ so that the resulting structure is o-minimal. 
Similarly, we will begin by considering expansions of $\CC_p$ by certain collections of analytic functions. 

We first study the existence of non-trivial field derivations $\delta:\CC_p\to\CC_p$ that respect a collection of functions $\mathcal{F}$ such that the structure $(\CC_p,\mathcal{F})$ is 1-h-minimal. 

\begin{thm}[Existence of non-trivial derivations, {{{Theorem~\ref{thm:derivations.text}}}}]\label{thm:derivations}
Let $\cT$ be a $1$-h-minimal theory and let $K$ be a model of $\cT$. Then there exists a collection of commuting derivations $\Delta=(\partial_\tau)_{\tau\in \cB}$ on $K$ with the following properties:
\begin{enumerate}
    \item if $U\subset K^n$ is open and $h: U\to K$ is a $\emptyset$-definable $C^1$ function, then for every $a = (a_1, \ldots, a_n)$ and every $\partial_\tau$ we have
    \[
    \partial_\tau (h(a)) = \sum_{i=1}^n (\partial_i h)(a) \partial_\tau a_i.
    \]
    \item  We have $\bigcap_\tau \operatorname{ker} \partial_\tau = \operatorname{acl}(\emptyset)$, where $\mathrm{acl}$ denotes the model-theoretic algebraic closure operator.
\end{enumerate}
\end{thm}

Once we have these derivations, we study the model-theoretic question of local definability, following the treatment of Wilkie \cite{wilkie:holo} for holomorphic functions. 
We note that a few obstructions need to be sorted, one having to do with choosing appropriate reducts of 1-h-minimal theories which preserve 1-h-minimality while having a countable language (in general, reducts do not preserve 1-h-minimality, see \cite[Remark 5.1.2]{CHRV}), but the main one being that there is no non-Archimedean analog of Gabrielov's theorem which shows that the projections of analytic sets remain analytic. 

After this, the rest of the paper focuses on two specific $p$-adically analytic functions, which we show are part of 1-h-minimal theories in Corollary \ref{cor:countable.language.Cpexp} (even in a countable language).
We work in $\CC_p$, we denote the valuation ring by $\cO_{\CC_p}$, the maximal ideal by $\cM_{\CC_p}$ and the valuation by $v: \CC_p\to \QQ$.
One of the most famous transcendental functions known to be part of an o-minimal expansion of $\mathbb{R}$ is the real exponential function, so the first function we will focus on is $p$-adic exponentiation. 
We denote by $\DD_p\subset \CC_p$ the open ball of valuative radius $1/(p-1)$ around zero, and let $\exp: \DD_p\to 1+\DD_p$ be the $p$-adic exponential map. 

Using that $p$-adic exponentiation is part of a 1-h-minimal expansion of $\CC_p$, we obtain some results concerning the $p$-adic version of Schanuel's conjecture from transcendental number theory (see \S\ref{subsec:p-adicschanuel} for the statement of the conjecture). 
We first show that this conjecture implies a uniform version of itself (the analogous result about real exponentiation and Schanuel's conjecture is due to Kirby and Zilber \cite{kirby-zilber:unifschanuel}).
\begin{thm}
\label{thm:unifpschanuel}
    Let $V\subseteq\GG_a^n\times \GG_m^n$ be an irreducible algebraic variety defined over $\overline{\QQ}$ with $\dim V<n$. 
    Then the $p$-adic Schanuel conjecture implies that there exist finitely many proper $\QQ$-linear subspaces $L_1,\ldots,L_m\subset\GG_a^n$ such that if $(\mathbf{z},\exp(\mathbf{z}))\in V(\CC_p)$, then $\mathbf{z}\in L_i$, for some $i\in\{1,\ldots,n\}$. 
\end{thm}
Then we show two weak forms of the $p$-adic Schanuel conjecture using Ax's theorem on functional transcendence of exponentiation (although, as we explain in \S\ref{subsec:p-adicschanuel}, these two weak forms could have been already deduced from earlier work of Kirby~\cite{kirby-expalg}). 

Many applications of o-minimality to problems in arithmetic geometry (especially when it comes to problems of unlikely intersections) often involve periodic holomorphic functions, such as complex exponentiation, the exponential maps of abelian varieties, or automorphic functions of Shimura varieties. 
Some of the key applications of o-minimality in this area (like the o-minimal proofs of Ax--Schanuel type theorems) exploit the fact that these functions are not definable in an o-minimal structure (because of periodicity), but their restrictions to a fundamental domain are definable and one can move between fundamental domains in an algebraic (hence definable) way. 
This then presents a basic obstruction in the case of $p$-adic exponentiation because $\exp$ is injective on its domain of convergence and has no natural extension to all of $\CC_p$. 

For this reason, for the remainder of the paper we move our attention to our second function of interest: Tate's uniformization map of elliptic curves, which will be recalled in \S\ref{sec:Tate}. 
For $q\in \CC_p^\times$ we denote by $\phi_q: \CC_p^\times \to E_q$ the Tate uniformization map. 

We first prove an analog of a theorem of Kirby on so-called blurrings of the complex exponential function~\cite{kirby2019blurred}. 
This requires a functional transcendence result for Tate's uniformization map, which we obtain from works of Ax and Kirby, see Proposition \ref{prop:ax.schanuel.tate}. 
Notably, Kirby uses this result on blurrings to prove that the expansions of $\CC$ by various blurrings of the exponential are quasiminimal structures, and we explain that the same should hold for expansions of (the field structure of) $\CC_p$ equipped with appropriate blurrings of the Tate uniformization, see . 

\begin{thm}\label{thm:tateblurring}
    Let $Q$ be a $p$-adically dense subset of $E_q(\mathbb{C}_p)$ and let $V\subset \GG_m^n\times E_q^n$ be an irreducible rotund variety.
    Then $V(\CC_p)$ has a $p$-adically dense set of points of the form
    \[
    (\mathbf{x},\phi_q(\mathbf{x}) + \mathbf{r})
    \]
    where $\mathbf{x}\in\mathbb{G}_m^n(\CC_p)^n$ and $\mathbf{r}\in Q^n$.
\end{thm}

Finally, we establish a result on the $p$-adic density of likely intersections for powers of elliptic curves, following the work of \cite{eterovic-scanlon}.

\begin{thm}
    \label{thm:likely}
Let $V\subseteq E_q^n$ be an irreducible subvariety. 
Suppose $S$ is an algebraic subgroup of $E_q^n$ such that for every algebraic subgroup $T\subseteq E_q^n$ the quotient map $\psi:E_q^n\to E_q^n/T$ satisfies $\dim\psi(V)+\dim\psi(S)\geq n-\dim T$. 
Then for every p-adically dense subset $A\subseteq E_q^n(\CC_p)$,
\begin{equation*}
    \bigcup_{\mathbf{g}\in A}V\cap (\mathbf{g}+S)
\end{equation*}
is $p$-adically dense in $V$.
\end{thm}


The reader familiar with applications of o-minimality to problems in arithmetic geometry is probably aware of at least two major results in o-minimality: the Pila--Wilkie point-counting theorem \cite{pila-wilkie} and the Definable Chow theorem (also known as the o-minimal GAGA) \cite[Theorem 5.1]{peterzil-starchenko:chow}. 
These results will not make an appearance in our paper, but there are already non-Archimedean analogs and applications of both using earlier notions of tame p-adic analytic geometry \cite{cantoral-nguyen-stout-vermeulen,chambertloir-loeser,cluckers-comte-loeser,oswal:definablechow}.
We believe our results already illustrate some general principles about working with 1-h-minimal structures, and we expect that these ideas will be combined with those in the works cited above to lead to further applications of 1-h-minimality to problems of unlikely intersections in the non-Archimedean setting. 
In particular, we believe our results extend to more general uniformization maps of abelian varieties.

\subsection{Structure of the paper}

\begin{enumerate}
    \item[\S\ref{sec:prelims}:] We review the definitions of $p$-adic exponentiation and Tate's uniformization for elliptic curves. 
    We also give background on 1-h-minimality, Weierstrass systems, and we prove that $p$-adic exponentiation and Tate's uniformization are part of the same 1-h-minimal expansion of $\CC_p$. 

    \item[\S\ref{sec:derivations}:] We show the existence of field derivations on $\CC_p$ which respect the definable functions of a given 1-h-minimal expansion. 
    We study local definability, following Wilkie~\cite{wilkie:holo}.
    We then look at show some results concerning the $p$-adic Schanuel conjecture, including Theorem \ref{thm:unifpschanuel}. 

    \item[\S\ref{sec:blurredtate}:] We show how to obtain a functional transcendence result for Tate's uniformization map. 
    We then prove Theorem \ref{thm:tateblurring}, and use it to deduce quasiminimality.

    \item[\S\ref{sec:likely}:] We give background for the Zilber--Pink conjecture and we prove Theorem \ref{thm:likely}.
\end{enumerate}

\subsection*{Acknowledgements} The authors thank Raf Cluckers, Martin Hils, and Vincenzo Mantova for interesting discussions around this paper. 
We thank Tom Scanlon for suggesting the use of the Tate uniformization map.
We thank Abhishek Oswal for explaining the proof of Proposition~\ref{prop:dim.intersection} to us.

\section{Preliminaries}
\label{sec:prelims}

\subsection{\texorpdfstring{$p$}{p}-adic exponentiation}
\label{subsec:pexp}
We recall some generalities about the $p$-adic exponential map. 
Consider the usual power series defining the exponential
\[
\sum_{n=0}^\infty \frac{x^n}{n!}.
\]
On $\CC_p$, it converges on the open ball $\DD_p$ of valuative radius $1/(p-1)$ around $0$ and we denote the corresponding function by $\exp: \DD_p\to \CC_p$. 
The image of $\exp$ is precisely the open ball $1+\DD_p$ of valuative radius $1/(p-1)$ around $1$, and in fact $\exp$ is a bijection $\DD_p\to 1+\DD_p$. 
Its inverse is the familiar map $\log: 1+\DD_p\to \DD_p$, which is also analytic. 
The map $\exp$ is differentiable on its domain with $\exp' = \exp$, and satisfies the usual rule that $\exp(x+y) = \exp(x)\exp(y)$.


\subsection{Tate uniformization}\label{sec:Tate}

As we disussed in the introduction, one issue with the $p$-adic exponential is that its domain is the bounded set $\DD_p$ in $\CC_p$, and there is no natural way to extend $\exp$ to a global analytic function on all of $\CC_p$. 
So we will consider the Tate uniformization map for elliptic curves over $\CC_p$. 
All of the facts below can be found in~\cite[Ch.\,V]{silverman}. 

Note that $\CC_p$ has no infinite discrete subgroups, and so there is no analog of a lattice here.
Hence, there is no natural analog of the Weierstrass $\wp$-function.
However, if $\tau\in \CC$ is such that $\Im \tau > 0$, then we can decompose the corresponding $\wp$-function as
\[
\wp: \CC / (\ZZ+\tau \ZZ)\xrightarrow{x\mapsto \exp(2\pi ix)}\CC^\times / (q^\ZZ)\to E,
\]
where $q = \exp(2\pi i \tau)$ and $E$ is the elliptic curve associated with $\tau$.
In $\CC_p$ this second map does make sense and is called the Tate uniformization map.

In more detail, for $q\in \CC_p^\times$ with $v(q)>0$ and $k\in \NN$ consider the series
\[
s_k(q) = \sum_{n\geq 1}\frac{n^k q^n}{1-q^n}.
\]
This series has coefficients in $\ZZ$, and so, since $v(q)>0$, it converges to an element of $\CC_p$. 
Define $a_4(q) = -5s_3(q)$ and $a_6(q) = -(5s_3(q) + 7s_5(q))/12$ and let $E_q$ be the elliptic curve over $\CC_p$ defined by 
\[
y^2 + xy = x^3 + a_4(q)x + a_6(q).
\]
Consider the series
\begin{align*}
X(q,u) &= \sum_{n\in \ZZ} \frac{q^nu}{(1-q^nu)^2} - 2s_1(q), \\
Y(q,u) &= \sum_{n\in \ZZ} \frac{(q^nu)^2}{(1-q^nu)^3} + s_1(q),
\end{align*}
both of which converge for $u\in \CC_p\setminus q^\ZZ$, where $q^\ZZ := \{q^n\mid n\in \ZZ\}$. 
Then \emph{Tate uniformization} is the map
\[
\phi_q: \CC_p^\times \to E_q(\CC_p): u\mapsto (X(q,u), Y(q,u)),
\]
which is a surjective homomorphism of groups, with kernel $q^\ZZ$.

Note that the $j$-invariant of $E_q$ satisfies
\[
v(j(E_q)) = v\left( \frac{1}{q} + 744 + 196884q + \ldots \right) = -v(q) < 0.
\]
Therefore, not every elliptic curve over $\overline{\QQ}$ is isomorphic to an elliptic curve of the form $E_q$.
In particular, if $E$ is an elliptic curve over $\overline{\QQ}$ with complex multiplication, then it is well-known that $j(E)$ is an algebraic integer, and so $v(j(E)) \geq 0$.
Hence $E_q$ never has complex multiplication.

The functions $X,Y$ are differentiable with respect to $u$ and satisfy
\[
uX'(q,u) = X(q,u)+2Y(q,u).
\]
In particular, for a fixed $q$ the map $u\mapsto X(q,u)$ satisfies the following differential equation
\begin{equation}
\label{eq:diffeqtate}
    (uX')^2 = 4X^3 + X^2 + 4a_4(q)X + 4a_6(q).    
\end{equation}
This differential equation will be important later to obtain an Ax--Schanuel result for the Tate map, see Proposition~\ref{prop:ax.schanuel.tate}.

We will show that the restriction of $\phi_q$ to the fundamental domain $\{u\in \cO_{\CC_p}\mid v(u)<v(q)\}$ is definable in a suitable 1-h-minimal structure on $\CC_p$. 
Note that by the periodicity of $\phi_q$, the restriction to this set determines the entire function. 
Just as in the o-minimal case, one cannot expect $\phi_q$ to be definable on its entire domain $\CC_p$, precisely because of this periodicity. 
In fact, we will prove the more general fact that $\phi_q$ is definable in a 1-h-minimal structure on $\CC_p$ as a function of two variables $(q,u)$, similar to a result by Peterzil--Starcheko~\cite{peterzil-starchenko} for the Weierstrass $\wp$-function.

\subsection{Hensel minimality}\label{sec:hmin}

In this section, we recall the relevant results from h-minimality. 
We begin with some set-up and notation.

Let $K$ be a valued field with valuation ring $\cO_K$ and valuation $v: K\to \Gamma\cup\{\infty\}$ with additive notation. 
For $a\in K$ and $\lambda\in \Gamma$ denote by $B_{>\lambda}(a) = \{x\in K\mid v(x-a)>\lambda\}$ the open ball of radius $\lambda$ around $a$, and by $B_{\geq \lambda}(a) = \{x\in K\mid v(x-a)\geq \lambda\}$ the closed ball of radius $\lambda$ around $a$. 
If $B$ is an open ball, then we denote by $\rad B$ the radius of $B$. 
If $C$ is a finite subset of $K$, then a \emph{ball $1$-next to $C$} is a maximal open ball disjoint from $C$, where maximal is meant with respect to inclusion. 
If $\lambda$ is an element of $\Gamma$ with $\lambda \geq 0$ then a \emph{ball $\lambda$-next to $C$} is an open ball $B$ that is contained in an open ball $B_0$ $1$-next to $C$, and for which $\rad B = \rad B_0 + \lambda$. 
Note that the balls $\lambda$-next to $C$ partition the set $K\setminus C$, and two elements $x,y$ of $K\setminus C$ are in the same ball $\lambda$-next to $C$ if and only if for every $c\in C$ we have $v(x-y) > \lambda + v(x-c)$.
If $\lambda = v(m)$ for some integer $m$, we say simply \emph{$m$-next} rather than $v(m)$-next.

Hensel minimality, or h-minimality for short, is a tameness notion for valued fields developed in~\cite{CHR} in equicharacteristic zero, and in~\cite{CHRV} in mixed characteristic. 
It fulfills a similar role as o-minimality in the Archimedean context. 
We will work with a $1$-h-minimal theory $\cT$ in a language $\cL$ expanding the language of valued field $\cL_{\val} = \{0,1,+,\cdot, \cO\}$. 
The precise definition is not important for us, but we list the main properties that we use below. 
Let $K$ be a model of $\cT$.

\begin{enumerate}
    \item \textbf{(Jacobian property and differentiation~\cite[Cor.\,3.1.3]{CHRV})} If $f: K\to K$ is a definable function, then there exists a finite set $C\subset K$ and a positive integer $m$ such that on each ball $B$ $m$-next to $C$, $f$ is $C^1$, $v(f')$ is constant on $B$, and for all $x,y\in B$ we have that
    \[
        v(f(x)-f(y)) = v(f'(x)) + v(x-y).
    \]
    Moreover, the set $C$ is definable over the same parameters as $f$. 
    This is the analog of the monotonicity theorem from o-minimality.
    
    \item \textbf{(Algebraic Skolem functions~\cite[Prop.\,4.3.3]{CHR})} There exists an expansion of the language $\cL'\supset \cL$ such that $\Th_{\cL'}(K)$ is still $1$-h-minimal and such that algebraic Skolem functions exist. 
    That is, for any model $K'$ of $\Th_{\cL'}(K)$, and any subset $A\subset K'$, we have $\acl_{K', \cL'}(A) = \dcl_{K', \cL'}(A)$ (where $\mathrm{acl}$ and $\mathrm{dcl}$ denote the model theoretic algebraic closure and definable closure operators respectively). 
    Equivalently, every definable finite-to-one map has a definable section. 
    Moreover, algebraic closure does not grow in the sense that $\acl_{K', \cL}(A) = \acl_{K',\cL'}(A)$. 
    The existence of algebraic Skolem functions is a weak form of definable choice, and should be compared with definable choice in the o-minimal setting.
    
    \item \textbf{(Pregeometry~\cite[Lem.\,5.3.5]{CHR})} The model theoretic algebraic closure operator $\acl$ is a pregeometry on $K$. 
    In particular, combining this with the previous point we may expand the language to obtain that $\dcl$ is a pregeometry, while preserving $1$-h-minimality.
    
    \item \textbf{(Dimension theory~\cite[Prop.\,3.1.1]{CHRV})} To every non-empty definable set $X\subset K^n$ one can associate a positive integer called the \emph{dimension of $X$}, denoted by $\dim X$, with the following properties:
    \begin{enumerate}
        \item $\dim (X\cup Y) = \max\{\dim X, \dim Y\}$ for $X,Y\subset K^n$ definable,
        \item $\dim X > 0$ if and only if $X$ is infinite,
        \item if $f: X\to Y$ is definable and surjective with fibers of dimension $d$, then $\dim X = d + \dim Y$,
        \item if $X\subset K^n$ has dimension $n$ then $X$ has non-empty interior, and
        \item if $f: X\to Y$ is definable then for every integer $i$ the set of $y\in Y$ with $\dim f^{-1}(y) = i$ is a definable subset of $Y$.
    \end{enumerate}
    For $x\in X$ we denote by $\dim_x(X)$ the \emph{local dimension of $X$ at $x$}, which is by definition the minimum of $\dim(X\cap B)$ where $B$ ranges over all open balls in $K^n$ containing $x$.
    We then, moreover, have the following:
    \begin{enumerate}[resume]
        \item if $K$ is algebraically closed and $V\subset \mathbb{A}^N_K$ is an affine algebraic variety of dimension $n$ as an algebraic variety, then $V(K)$ has local dimension $n$ at every point. 
    \end{enumerate}
\end{enumerate}

Since we are interested in $p$-adic exponentiation and Tate uniformization, one would hope that the structure $\CC_p$ is 1-h-minimal in the language of valued fields with a symbol for $\exp$ or $\phi_q$. 
At first glance, this seems to follow from~\cite[Thm.\,6.2.1]{CHR} which asserts that fields equipped with separated analytic structure are $1$-h-minimal, and there are certainly separated Weierstrass systems containing the exponential and Tate uniformization. 
However, there is a subtle issue with this, as reducts of h-minimal structures are not necessarily h-minimal, and so one cannot deduce from the h-minimality of $\CC_p$ equipped with separated analytic structure that also $\CC_p$ with only $\exp$ or $\phi_q$ is h-minimal. 

While one could just work in a language with symbols for elements of a separated Weierstrass system, we will also need our language to be countable. 
Our solution is a general tool to get a countable language for h-minimal fields. 
For the proof we assume some knowledge about h-minimal structures.

\begin{lem}\label{lem:h.min.countable.language}
    Let $\cL$ be a language expanding the language of valued fields $\cL_{\mathrm{val}} = \{+,\cdot, \cO\}$ and let $K$ be an $\cL$-structure for which $\Th_{\cL}(K)$ is 1-h-minimal. 
    If $\cL_0$ is a countable sublanguage of $\cL$, then there exists a countable language $\cL'$ expanding $\cL_0$ such that $K$ has $\cL'$-structure for which $\Th_{\cL'}(K)$ is 1-h-minimal, and such that every $\cL'$-definable set is also $\cL$-definable.
\end{lem}

\begin{proof}
    Let $\Gamma_K$ be the value group of $K$. 
    For $\lambda\in \Gamma_K^\times, \lambda\geq 0$ let $\RV_\lambda$ be the leading term structure $K^\times/(1+B_{>\lambda}(0)) \cup \{0\}$ and let $\rv_\lambda: K\to \RV_\lambda$ be the natural map (extended by $\rv_\lambda(0) = 0$). 

    We will inductively construct countable languages $\cL_0\subset \cL_1\subset \cL_2, \ldots$ such that $\cL' = \cup_i \cL_i$ works, and such that moreover any $\cL_i$-definable set is also $\cL$-definable. 
    This will be done by witnessing preparation data from the definition of 1-h-minimality. 
    
    Suppose that $\cL_i$ has been constructed already. 
    For each $\cL_i$-formula $\phi(x, \alpha, \beta; y)$ where $x$ is in $K$, $\alpha$ is in $\RV^k$ (for some $k\in \NN$), $\beta$ is in $\cup_{\lambda\geq 0} \RV_\lambda$, $y$ is in $K^n$ (for some $n\in \NN$), we will add a predicate witnessing the uniform preparation of $\phi$. Note that there are only countably many such formulas. By 1-h-minimality in the language $\cL$, uniform preparation~\cite[Prop.\,2.3.2]{CHRV} implies that there is an $\cL$-formula $\psi(x; y)$ such that for each $y\in K^n$, $\psi(K; y)$ is a finite set uniformly preparing $\phi(x, \alpha, \beta; y)$. We now add a predicate for the formula $\psi$ to the language $\cL_i$. Continuing in this way gives a countable language $\cL' = \cup_i \cL_i$, and we claim that it suffices.

    We verify that $K$ is 1-h-minimal as an $\cL'$-structure. So let $K'$ be elementarily equivalent to $K$ in the language $\cL'$, let $\lambda\geq 0$ be in $\Gamma_{K'}$, let $A\subset K'$, let $B\subset \RV_{K'}$, let $\xi\in \RV_{K', \lambda}$, and let $X$ be $\cL'(A\cup B\cup \{\xi\})$-definable. We can assume that $A$ and $B$ are finite. Then there is some integer $i$ such that $X$ is already $\cL_i(A\cup B \cup \{\xi\})$-definable, and so there is an $\cL_i$-formula $\phi(x,\alpha, \beta; y)$ such that $\phi(K', B, \xi; A)$ is equal to $X$. Since preparation is first order by~\cite[Lem.\,2.4.2]{CHR}, our construction shows that there is some $\cL_{i+1}(A)$-definable set which $\lambda$-prepares $X$, as desired.
\end{proof}

\subsection{Weierstrass systems}\label{sec:weierstrass}

We recall some generalities about Weierstrass systems, and refer to~\cite{Lip, LipRob,CL} for more details.

If $R\subset \cO_{\CC_p}$ is a subring, we denote
\[
R\langle \xi \rangle = \left\{\sum_i a_i \xi^i\mid a_i\xrightarrow[|i|\to \infty]{} 0 \right\}
\]
for the ring of \emph{strictly convergent power series}, where $\xi = (\xi_1, \ldots, \xi_m)$. 
Every element of $R\langle \xi\rangle$ is a convergent power series on $\cO_{\CC_p}^m$ and hence naturally defines a function $\cO_{\CC_p}^m\to \CC_p$.
An element $f$ of $\cO_{\CC_p}\langle \xi\rangle$ is called \emph{regular in $\xi_m$ of degree $d$} if 
\[
f = \xi_m^d + a_1\xi_m^{d-1} + \ldots + a_d + g,
\]
where $a_1, \ldots, a_d\in \cO_{\CC_p}$, and $g\in \cO_{\CC_p}\langle \xi\rangle$ has all of its coefficients in $\cM_{\CC_p}$.

A \emph{Weierstrass system} consists of a system of rings $\cA=\{A_m\}_{m\in \NN}$ such that 
\begin{enumerate}
    \item $A_m\subset \cO_{\CC_p}\langle \xi_1, \ldots, \xi_m\rangle$ for every $m$,
    \item $A_{m-1}[\xi_m]\subset A_m$ for every $m\geq 1$,
    \item for $1\leq k\leq m$ and $f\in A_m$, if we write $f = \sum_{i} f_i(\xi_1, \ldots, \xi_k)(\xi_{k+1}, \ldots, \xi_m)^i$, then the $f_i$ are in $A_k$, and
    \item if $f\in A_m$ is regular of degree $d$ in $\xi_m$ then for every $g\in A_m$ there exist unique $q\in A_m$ and $r\in A_{m-1}[\xi_m]$ of degree at most $d-1$ such that $g=qf+r$.
\end{enumerate}
This last property is precisely Weierstrass division.
By~\cite[(3.2)]{LipRob} the full system $A_m = \cO_{\CC_p}\langle \xi_1, \ldots, \xi_m\rangle$ is a Weierstrass system.
Also, the intersection of any number of Weierstrass systems is a Weierstrass system. 
Therefore, if $A\subset \bigsqcup_m \cO_{\CC_p}\langle \xi_1, \ldots, \xi_m\rangle$ is any subset, there is a smallest Weierstrass system $\cA$ containing $A$, which we call the \emph{Weierstrass system generated by $A$}.
More explicitly, $\cA$ is obtained from the elements of $A$ by repeatedly closing under the operations above.

If $\cA = \{A_m\}_{m\in\NN}$ is a Weierstrass system, then define $\cL_{\cA} = \{+,-,\cdot,/,0,1,\cO\}\cup\bigcup_m A_m$.
By interpreting the elements of $A_m$ as function symbols on $\cO_{\CC_p}^m$, $\CC_p$ naturally has $\cL_{\cA}$-structure.
Moreover, by~\cite[Thm.\,4.5.15]{CL} $\CC_p$ has quantifier elimination in $\cL_{\cA}$.

Strictly convergent power series converge on the closed ball $\cO_{\CC_p}^m$, while $\exp$ and $\phi_q$ are defined on open balls.
Therefore, to prove that $\exp$ and $\phi_q$ are definable in a $1$-h-minimal structure on $\CC_p$, we will also need some basic separated Weierstrass system. 
Write $\rho = (\rho_1, \ldots, \rho_n)$.
Following~\cite{Lip}, we then define
\[
S_{m,n} = \CC_p\otimes_{\cO_{\CC_p}}\bigcup_{B\in \cB} B\langle \xi\rangle \llbracket \rho \rrbracket,
\]
where $\cB$ is the collection of all rings of the form
\[
\ZZ_p[a_i\mid i\in \NN]_{\{a\in \ZZ_p[a_i]\,\mid\, v(a)=0\}}^{\wedge},
\]
where $(a_i)_i$ is a sequence in $\cO_{\CC_p}$ converging to $0$, and $\phantom{}^{\wedge}$ denotes the $p$-adic completion.
Note that $\cO_{\CC_p}\langle \xi\rangle$ is contained in $S_{m,0}$.
Every element $f$ of $S_{m,n}$ is a power series, and naturally determines a function
\[
f: \cO_{\CC_p}^m\times \cM_{\CC_p}^n \to \CC_p: (\xi, \rho)\mapsto f(\xi, \rho).
\]
Therefore, $\CC_p$ naturally has $\cL_S$-structure, where $\cL_S = \{+,-,\cdot,/,0,1,\cO\}\cup \bigcup_{m,n} S_{m,n}$.
By~\cite[Thm.\,4.2]{LipRob}, the $\cL_S$-theory of $\CC_p$ has quantifier elimination, while~\cite[Cor.\,6.2.7]{CHR} implies that this structure is $1$-h-minimal.


Using Lemma~\ref{lem:h.min.countable.language} and the above language $\cL_S$, we will find a countable language on $\CC_p$ in which both the exponential and the Tate uniformization map are definable.

\begin{cor}\label{cor:countable.language.Cpexp}
    There exists a countable language $\cL\supset \cL_{\val}$ and $\cL$-structure on $\CC_p$ such that 
    \begin{enumerate}
        \item $\Th_{\cL}(\CC_p)$ is 1-h-minimal,
        \item $\exp: \DD_p\to 1+\DD_p$ is $\cL$-definable, and
        \item the map
        \[
        \Phi: \{(q,u)\in \CC_p^\times \times \CC_p\mid v(q)>v(u)\geq 0\}\to \CC_p^2: (q,u)\mapsto (X(q,u), Y(q,u))  
        \]
        is $\cL$-definable.
    \end{enumerate}
\end{cor}

\begin{proof}
    Take any $\alpha\in \overline{\ZZ}\cap \CC_p$ with $v(\alpha) = 1/(p-1)$, where $\overline{\ZZ}\subset \overline{\QQ}$ is the set of algebraic integers. 
    Let $L$ be a finite extension of $\QQ_p$ containing $\alpha$, with ring of integers $\cO_L$.
    Then all coefficients of the power series
    \[
    f:=\sum_{n} \frac{\alpha^n}{n!}x^n
    \]
    are in the ring $\cO_L\subset \cO_{\CC_p}$.
    Since this ring is finitely generated over $\ZZ_p$, this power series is in $S_{0,1}$.
    Now simply note that for $x\in \DD$ we have $\exp(x) = f(x / \alpha)$ and so $\exp$ is $\cL_S$-definable.
    
    We show that also Tate uniformization is $\cL_S$-definable. 
    By~\cite[p.\,426]{silverman}, we have that if $-v(q)<v(u)<v(q)$ then
    \[
    X(q,u) = \frac{u}{(1-u)^2} + \sum_{d\geq 1}\left(\sum_{m\mid d}m(u^m+u^{-m}-2) \right)q^d.
    \]
    Define the power series
    \[
    F(q,u) = \sum_{d\geq 1}\left(\sum_{m\mid d}m u^m\right)q^d.
    \]
    To show that $X$ is $\cL_S$-definable on the specified domain, it is then enough to show that $(q,u)\mapsto F(q,u)$ is $\cL_S$-definable on the same domain. Now note that $F(q,u/q)$ is a power series in $q,u$ with coefficients in $\ZZ$, and so it is contained in $S_{0,2}$. But then the map $(q,u)\mapsto F(q,u)$ is $\cL_S$-definable on the domain $\{(q,u) \mid v(q)>0, v(u)>-v(q)\}$. A similar argument applied to $F(q,u^{-1})$ yields that $(q,u)\mapsto F(q,u)$ is $\cL_S$-definable on $\{(q,u)\mid v(q)>0, v(q)>v(u)>-v(q)\}$. The $\cL_U$-definability of $Y(q,u)$ follows from a similar argument, using that for $v(q)>v(u)>-v(q)$ we have that
    \[
    Y(q,u) = \frac{u^2}{(1-u)^3} + \sum_{d\geq 1}\left(\sum_{m\mid d}\frac{(m-1)m}{2}u^m - \frac{m(m+1)}{2}u^{-m}+m\right) q^d.
    \]
    We conclude that $\Phi$ is $\cL_S$-definable.

    To finish the proof, apply Lemma~\ref{lem:h.min.countable.language} to the language $\cL_0 = \cL_{\val}\cup\{\exp, \Phi\}$.
\end{proof}

If one wants to focus on a single Tate uniformization $\phi_q$ for some $q\in \CC_p$, then one can simply add a symbol for $q$ to the language $\cL$ of the above lemma.
Indeed, adding constant symbols preserves 1-h-minimality, and $\phi_q$ is clearly definable on its fundamental domain in the expanded structure.



\section{Existence of Derivations}
\label{sec:derivations}
The goal of this section is to show that $1$-h-minimal fields admit many non-trivial derivations and to prove Theorem~\ref{thm:derivations}. 
For this we adapt a method used in \cite{bays-kirby-wilkie} for building derivations on $\CC$ respecting the usual complex exponential function that is based on the o-minimality of $(\RR,\exp)$. The existence of derivations on h-minimal fields may be useful in other contexts.

\begin{defn}
    Let $F$ be a field. 
    A \emph{derivation on $F$} is a function $\partial:F\to F$ that satisfies the following properties:
    \begin{enumerate}
        \item For all $a,b\in F$ we have $\partial(a+b) = \partial(a) + \partial(b)$.
        \item For all $a,b\in F$ we have $\partial(ab) = a\partial(b) + b\partial(a)$. 
    \end{enumerate}
    A \emph{differential field} is a pair $(F,\Delta)$, where $
    \Delta$ is a set of derivations on $F$.

    The set $\bigcap_{\partial\in\Delta}\ker \partial$ is called the \emph{field of constants}.
\end{defn}

\begin{defn}
    Let $(F, \Delta)$  be a differential field of characteristic zero. 
    Suppose that $F$ is a topological field, let $U\subset F^n$ be an open subset, and let $f:U\to F$ be a differentiable function. 
    We say that $\partial\in\Delta$ \emph{respects $f$} if for every $x\in U$ we have 
    \begin{equation*}
        \partial(f(x)) = \sum_{i=1}^n (\partial_i f)(x)\partial x_i,
    \end{equation*}
    where $\partial_i f$ denote the partial derivative of $f$ with respect to the $i$-th variable. 
\end{defn}

Of course, the zero derivation (the trivial derivation) respects any function. 
With this terminology, Theorem~\ref{thm:derivations} may be restated as follows.

\begin{thm}[Existence of non-trivial derivations]\label{thm:derivations.text}
Let $\cT$ be a $1$-h-minimal theory and let $K$ be a model of $\cT$. 
Then there exists a collection of commuting derivations $\Delta=(\partial_\tau)_{\tau\in \cB}$ on $K$ with the following properties:
\begin{enumerate}
    \item \label{it:derivations.respect} if $U\subset K^n$ is open and $h: U\to K$ is a $\emptyset$-definable $C^1$ function, then every $\partial_\tau$ respects $h$, and
    \item \label{it:kernel.derivations} $\bigcap_\tau \operatorname{ker} \partial_\tau = \operatorname{acl}(\emptyset)$.
\end{enumerate}
\end{thm}

For the proof, we will use the basic properties of h-minimal theories as recalled in Section~\ref{sec:hmin}.

\begin{proof}
First, we may expand our language to $\cL'\supset \cL$ by an $\RV$-expansion such that $\acl = \dcl$ in $\Th_{\cL'}(K)$ and such that $K$ is $1$-h-minimal. Since $\acl$ is a pregeometry, and since $\acl=\dcl$, also $\dcl$ is a pregeometry on $K$. Let $\cB$ be a basis for the pregeometry $\dcl$. For $\tau\in \cB$ denote $\cB_\tau = \cB\setminus \{\tau\}$. For each $\tau\in \cB$, we will define a derivation $\partial_\tau$ on $K$ in the following way: for $a\in K$ take an $\cL'(B_\tau)$-definable function $f: X\subset K\to K$ for which $\tau\in X$ and $f(\tau) = a$, and we put $\partial_\tau(a) = f'(\tau)$. We claim that 
\begin{enumerate}
    \item such an $f$ always exists,
    \item that it is differentiable at $\tau$,
    \item that $\partial_\tau$ is independent of the choice of $f$, and
    \item that the resulting map $\partial_\tau$ is a derivation with the above properties.
\end{enumerate}

First let us show that such an $f$ exists. Since $\cB$ is a basis for the $\dcl$-pregeometry, there exists an $\cL'$-formula $\phi(x, y, z)$ and a tuple $v\in \cB_\tau^n$ such that $\phi(x,\tau, v)$ has a unique solution in $K$ which is $a$. Let $Y$ be the set of $y\in K$ such that there exists a unique $x$ satisfying $\phi(x,y, v)$. Then $Y$ is an $\cL'(\cB_\tau)$-definable set containing $\tau$. On $Y$, $\phi$ determines an $\cL'(\cB_\tau)$-definable function $f$ by mapping $y\in Y$ to the unique solution of $\phi(x, y, v)$. Note that $f(\tau) = a$. This proves the existence of such a function $f$.

Let us check that any such $f$ is differentiable at $\tau$. Using $1$-h-minimality, let $C\subset K$ be a finite $\cL'(\cB_\tau)$-definable set and let $m$ be an integer such that $f$ is $C^1$ on balls $|m|$-next to $C$. Then $\tau$ cannot be in $C$, since then $\tau$ would be in a finite $\cL'(\cB_\tau)$-definable set, which is impossible since $\cB$ is a $\dcl$-basis and $\acl_{\cL'}=\dcl_{\cL'}$ in $K$. So $f$ is differentiable at $\tau$, and hence it makes sense to define $\partial_\tau(a) = f'(\tau)$.

Next, we show that $f'(a)$ is independent of the choice of $f$. Suppose $g: Y'\to K$ is another $\cL'(\cB_\tau)$-definable function with $\tau\in Y'$ and $g(\tau) = a$. Then as above $g$ is differentiable at $\tau$. Let $C'$ be a finite $\cL'(\cB_\tau)$-definable set and $m$ be a positive integer such that $C$ $|m|$-prepares the $\cL'(\cB_\tau)$-definable set
\[
Z = \{y\in Y\cap Y'\mid f(y) = g(y)\}.
\]
Note that $\tau \in Z$. For the same reason as before, $\tau$ cannot be in $C'$, so there is an open ball containing $\tau$ which is completely contained in $Z$. But then $f$ and $g$ agree on a neighbourhood of $\tau$, so that $f'(\tau) = g'(\tau)$.

We prove that the $\partial_\tau$ commute.
So take $\tau\neq\sigma$ in $\cB$ and $a\in K$.
Write $\cB_{\tau, \sigma} = \cB\setminus \{\tau, \sigma\}$.
Then there exists an $\cL'(\cB_{\tau, \sigma})$-definable function $f: X\subset K^2\to K$ with $f(\tau, \sigma) = a$.
We claim that $(\tau, \sigma)$ is contained in the interior of $X$.
Suppose this is not the case.
Then $(\tau, \sigma)$ is contained in $\overline{(K^2\setminus X)}$ but not in $K^2\setminus X$.
But then~\cite[Prop.\,3.1.1(3.f)]{CHRV} shows that $(\tau, \sigma)$ is contained in a $1$-dimensional $\cL'(\cB_{\sigma, \tau})$-definable set, contradicting the fact that $\cB$ is a basis for the $\dcl$-pregeometry.

Let $Y\subset X$ be the set of $x\in X$ at which $f$ is not $C^2$, which is an $\cL'(\cB_{\tau, \sigma})$-definable set. 
Then~\cite[Prop.\,3.1.1]{CHRV} shows that $Y$ has dimension at most $1$, and a similar argument as above shows that $(\tau, \sigma)\notin Y$.
Therefore $f$ is $C^2$ at $(\tau, \sigma)$ and a straight-forward computation shows that
\begin{align*}
    \partial_\tau(\partial_\sigma(a)) =  (\partial_1 \partial_2 f)(\tau, \sigma)= (\partial_2 \partial_1 f)(\tau, \sigma) = \partial_\sigma(\partial_\tau(a)).
\end{align*}

We now check properties~(\ref{it:derivations.respect}) and~(\ref{it:kernel.derivations}) of the theorem. So let $h: U\to K$ be a $\emptyset$-definable $C^1$ function on an open $U\subset K^n$ and let $a=(a_1, \ldots, a_n)\in U$. Let $f_i: X_i\subset K\to K$ be $\cB_\tau$-definable with $\tau\in X$ and $f(\tau) = a_i$, for $i=1, \ldots, n$. 
Then the function
\[
g: K\to K: t\mapsto \begin{cases}
    h(f_1(t), \ldots, f_n(t)) & \text{ if } t\in \bigcap_i X_i, \\
    0 & \text{ else}
\end{cases}
\]
is $\cL'(\cB_\tau)$-definable and satisfies $g(\tau) = h(a)$.
Hence, by our construction
\[
\partial_\tau (h(a)) = g'(\tau) = \sum_{i=1}^n (\partial_i h)(a)\partial_\tau(a).
\]

For the final property, first take $a$ in $\acl_{\cL}(\emptyset) = \dcl_{\cL'}(\emptyset)$. Then we can simply take $f = a$ a constant function, and so $\partial_\tau(a) = 0$. Conversely, let $a$ be in $\ker \partial_\tau$, we show that $a$ is in $\dcl_{\cL'}(\cB_\tau) = \acl_{\cL}(\cB_\tau)$. We have an $\cL'(\cB_\tau)$-definable function $f: X\subset K\to K$ defined at $\tau$ with $f(\tau) = a$ and such that $f'(\tau) = \partial_\tau(a) = 0$. The set of $x\in X$ with $f'(x) = 0$ is an $\cL'(\cB_\tau)$-definable set, and hence so is its image under $f$. By dimension theory, this image is a finite set. But $a$ is contained in this image, which is a finite $\cL'(\cB_\tau)$-definable set, and so $a\in \acl_{\cL'}(\cB_\tau) = \acl_{\cL}(\cB_\tau)$.
\end{proof}

\begin{cor}\label{cor:derivation.Cp}
    Let $t_1, \ldots, t_n$ be in $\CC_p$ and let $q$ be in $\CC_p$ with $v(q) > 0$. Then there exists a non-trivial derivation $\partial$ on $\CC_p$ respecting $\exp$ and $\phi_q$ for which $\partial t_i = 0$ for all $i$.
\end{cor}

\begin{proof}
    Apply the previous theorem to $\CC_p$ equipped with $\cL(q, t_1, \ldots, t_n)$-structure, where $\cL$ is the countable language from Corollary~\ref{cor:countable.language.Cpexp}. 
\end{proof}

In \cite{kirby-expalg} Kirby defines and studies \emph{Khovanskii systems} of exponential polynomials over any exponential field (even when the exponential is not defined over the whole field). 
One of the main results of \cite{kirby-expalg} is that a notion of exponential closure defined in terms of these Khovanskii systems agree with the notion of closure coming from derivations on the field respecting the exponential. 
Therefore, we immediately get the following. 

\begin{cor}[see {{\cite[Theorem 1.1]{kirby-expalg}}}]
    The pregeometry of derivations agrees with the pregeometry of Khovanskii systems.
\end{cor}

On $\QQ_p$ and $\CC_p$, there can only be countably many zeros of a system of analytic equations, so one can use \cite[Remark 3.4 and Proposition 7.2]{kirby-expalg} (as well as some fixes provided in \cite{aek-differentialEC,aek-closureoperator}) to recover the existence of non-trivial derivations respecting $\exp$ through purely differential algebraic methods, without using 1-h-minimality.
However, these methods depends on Ax's famous functional transcendence result for exponentiation (which we will recall in \S\ref{subsec:p-adicschanuel}).
Such functional transcendence statements are not known for arbitrary functions definable in 1-h-minimal structures.

\subsection{Local definability}
\label{subsec:localdefinability}
Suppose we expand $\CC_p$ with a collection of analytic functions $\mathcal{F}$ so that the corresponding structure defines a 1-h-minimial theory. 
A natural question is: what other analytic functions are definable in $(\CC_p,\mathcal{F})$?
Following Wilkie in \cite{wilkie:holo}, we will answer this question at \emph{generic} points in Theorem \ref{thm:localdef}. 
We remark that further study about definability in the complex o-minimal setting at non-generic points was done in \cite{jkgs}. 

A function $F: U\subset \CC_p^n\to \CC_p$ on an open set $U$ is said to be \emph{analytic} if for every $a\in U$ there exists a closed box around $a$ such that on this box $F$ is described by a strictly convergent power series.

Given an analytic function $F$, we let $\mathrm{dom}(F)\subseteq\CC_p^n$ denote the domain of $F$, which is always an open subset of $\CC_p^n$. 

\begin{defn}
    Given an analytic function $F$, a \emph{proper restriction of $F$} is the restriction of $F$ a set $\Delta\subseteq\mathrm{dom}(F)$ of the form $\Delta = D_1\times\cdots\times D_n$, where each $D_i$ is a closed ball in $\CC_p$ whose centre is in $\overline{\QQ}$, so that $F\upharpoonright_{\Delta}$ can be expressed as a strictly convergent power series. 
    In this case we say that $\Delta$ is \emph{suitable} for $F$. 
\end{defn}

Given a collection $\mathcal{F}$ of analytic functions, we denote by $\cF^{\PR}$ the set of all proper restrictions of elements of $\cF$, and by $\left(\CC_p, \mathcal{F}^{\PR}\right)$ the corresponding expansion of $\CC_p$.
If $G$ is a proper restriction of $F\in \cF$ with domain $\Delta = D_1\times \ldots \times D_m$, then by definition there exist $a,b\in \overline{\QQ}^m$ and $c\in \overline{\QQ}$ such that the map
\[
(\xi_1, \ldots, \xi_m)\mapsto cF(a_1x_1 + b_1, \ldots, a_mx_m + b_m)
\]
is an element of $\cO_{\CC_p}\langle \xi\rangle$.
For each $m\in \NN$, let $A_m\subset \cO_{\CC_p}\langle \xi\rangle$ be all strictly convergent power series which can be obtained in this way.
Then there is a smallest Weierstrass system containing all of the $A_m$, which we denote by $\cA_{\cF}$.
In view of the discussion in Section~\ref{sec:weierstrass}, $(\CC_p, \cL_{\cA_\cF})$ has quantifier elimination, and is an expansion by definition of $(\CC_p, \cF^{\PR})$.

\begin{defn}
    Given a collection $\mathcal{F}$ of analytic functions and another analytic function $F$, we say that $F$ is \emph{locally definable from $\mathcal{F}$} if every proper restriction of $F$ is definable in $\left(\CC_p, \mathcal{F}^{PR}\right)$. 

    Equivalently, $F$ is locally definable from $\mathcal{F}$ if and only if for every $\mathbf{w}\in\mathrm{dom}(F)$ there is  $\Delta$ suitable for $F$ such that $\mathbf{w}\in\Delta$ and $F\upharpoonright_\Delta$ is definable in $\left(\CC_p, \mathcal{F}^{PR}\right)$
\end{defn}

\begin{defn}
    Given open sets $U\subseteq\CC_p^{n+1}$ and $V\subseteq\CC_p^n$ and analytic functions $F:U\to\CC_p$ and $f:V\to\CC_p$, we say that $f$ is \emph{implicitly defined} from $F$ if for all $\mathbf{w}\in V$ we have $(\mathbf{w},f(\mathbf{w}))\in U$ and $F(\mathbf{w},f(\mathbf{w}))=0\neq\frac{\partial F}{\partial z_{n+1}}(\mathbf{w},f(\mathbf{w}))$. 
\end{defn}

Given a collection $\mathcal{F}$ of analytic functions on $\CC_p$ which is closed under differentiation, we let $\widetilde{\mathcal{F}}$ denote the smallest collection of functions containing $\mathcal{F}$ and which is closed under composition and implicit definability. 
We note that $\cA_{\widetilde{\cF}} = \cA_{\cF}$, since Weierstrass systems are automatically closed under composition and implicit definition (see~\cite[Rem.\,4.5.2]{CL}).

\begin{defn}
    Given a collection $\mathcal{F}$ of analytic functions, we let $\mathrm{Der}_{\CC_p}(\mathcal{F})$ denote the set of all field derivations $\delta:\CC_p\to\CC_p$ which respect every function in $\mathcal{F}$ at every point of their corresponding domain. 
\end{defn}

\begin{lem}[{{cf.~\cite[Lemma 3.3]{wilkie:holo}}}]
\label{lem:derivationspaces}
For every collection $\mathcal{F}$ of analytic functions on $\CC_p$ which is closed under differentiation we have $\mathrm{Der}_{\CC_p}(\mathcal{F}) = \mathrm{Der}_{\CC_p}\left(\widetilde{\mathcal{F}}\right) = \mathrm{Der}_{\CC_p}\left(\cA_\cF\right)$. 
\end{lem}
\begin{proof}
    Since $\cA_\cF = \cA_{\widetilde{\cF}}$, it is enough to prove that $\mathrm{Der}_{\CC_p}(\cF) = \mathrm{Der}_{\CC_p}(\cA_\cF)$.
    One inclusion is obvious, so let $\delta\in \mathrm{Der}_{\CC_p}(\cF)$. 
    The construction of $\cA_\cF$ is obtained inductively from elements of $\cF^\PR$, so we prove the statement by induction on the operations for a Weierstrass system.
    For the base of the induction, we have that $\delta$ respects all elements from $A_m\subset \cO_{\CC_p}\langle \xi_1, \ldots, \xi_m\rangle$ constructed above, since $\overline{\QQ}$ is contained in $\ker \delta$.
    For the induction step it is enough to prove the following facts.
    \begin{enumerate}
        \item If $\delta$ respects $f,g\in \cO_{\CC_p}\langle \xi\rangle$ then it also respect $f+g$ and $fg$. 
        This is clear by definition of a derivation.
        \item If $\delta$ respects $f_0, \ldots, f_d\in \cO_{\CC_p}\langle \xi_1, \ldots, \xi_{m-1}\rangle$, then $\delta$ also respects $\sum_i f_i \xi_m^i\in \cO\langle \xi_1, \ldots, \xi_m\rangle$. 
        This is clear by the previous fact, using that $\delta$ respects the identity function.
        \item If $\delta$ respects $f\in \cO_{\CC_p}\langle \xi_1, \ldots, \xi_m\rangle$ and we write $f = \sum_i f_i(\xi_1, \ldots, \xi_k)(\xi_{k+1}, \ldots, \xi_m)^i$, then $\delta$ also respects each $f_i$.
        This follows by taking partial derivatives of $f$ and plugging in $\xi_{k+1} = \ldots = \xi_m = 0$.
        Note that this step uses that $\cF$ is closed under differentiation.
        \item Finally, let $f\in \cO_{\CC_p}\langle \xi_1, \ldots, \xi_m\rangle$ be regular of degree $d$ in $\xi_m$, and take $g\in \cO_{\CC_p}\langle\xi_1, \ldots, \xi_m\rangle$. 
        Then there exist unique $q\in \cO_{\CC_p}\langle \xi_1, \ldots, \xi_m\rangle$ and $r\in \cO_{\CC_p}\langle\xi_1, \ldots, \xi_{m-1}\rangle [\xi_m]$ of degree at most $d-1$ such that $g=qf+r$.
        We have to show that if $\delta$ respects $f$ and $g$, then it also respects $q$ and $r$.

        For $h\in \cO_{\CC_p}\langle \xi_1, \ldots, \xi_m\rangle$ denote by $h^\delta\in \cO_{\CC_p}\langle \xi_1, \ldots, \xi_m\rangle$ the element obtain from $h$ by applying $\delta$ to all coefficients of the power series.
        Then $h\mapsto h^\delta$ is a derivation on $\cO_{\CC_p}\langle \xi_1, \ldots, \xi_m\rangle$.
        Moreover, we have that for $a\in \cO_{\CC_p}^m$
        \[
        \delta(h(a)) = h^{\delta}(a) + \sum_{i=1}^m (\partial_i h)(a)\delta a_i.
        \]
        Therefore $\delta$ respects $h$ if and only if $h^\delta = 0$.

        Now since $g=qf+r$ we have that
        \[
        g^\delta = q^\delta f + qf^\delta +r^\delta.
        \]
        Since $\delta$ respects $f$ and $g$, we obtain that
        \[
        q^\delta f + r^\delta = 0.
        \]
        But $f$ is regular of degree $d$, while $r^\delta$ is a polynomial in $\xi_m$ of degree at most $d-1$.
        Hence we conclude that $r^\delta = q^\delta = 0$ and $\delta$ respects both $q$ and $r$.
        \qedhere

    \end{enumerate}
\end{proof}

\begin{defn}
    Given a collection of functions $\mathcal{F}$ and a subset $X\subseteq\CC_p$, we let $D_{\mathcal{F}}(X)$ denote the set of all elements of $\CC_p$ of the form $F(\mathbf{w})$, where $F\in\mathcal{F}$ and $\mathbf{w}$ is a tuple of elements of $X$ such that $\mathbf{w}$ is in the domain of $F$.  
    We also define $LD_{\mathcal{F}}(X)$ in a similar way, but now $F$ is allowed to be locally definable from $\mathcal{F}$. 
\end{defn}

\begin{lem}[cf.\ {{\cite[Theorem 2.2]{wilkie:holo}}}]
\label{lem:pregeometryoperators}
    For any collection $\mathcal{F}$ of analytic functions which is closed under differentiation we have that $LD_{\mathcal{F}}$ and  $D_{\widetilde{\mathcal{F}}}$ are pregeometries on $\CC_p$. 
\end{lem}

\begin{proof}
    Let us check that $D_{\widetilde{\cF}}$ is a pregeometry, the proof for $LD_\cF$ is similar.
    The only non-obvious property is the exchange principle.
    So take $a_1, \ldots, a_n\in \CC_p$ and $b\in D_{\widetilde{\cF}}(a_1, \ldots, a_n)\setminus D_{\widetilde{\cF}}(a_1, \ldots, a_{n-1})$.
    By definition, there exists a function $F\in \widetilde{F}$ with $(a_1, \ldots, a_n)$ in its domain such that
    \[
    b = F(a_1, \ldots, a_n).
    \]
    Now, if $(\partial_n^i F)(a_1, \ldots, a_n) = 0$ for all $i\geq 1$, then in a neighbourhood of $(a_1, \ldots, a_n)$, $F$ is independent of $a_n$.
    But then $b\in D_{\widetilde{\cF}}(a_1, \ldots, a_{n-1})$, which is impossible.
    Let $i$ be minimal such that $(\partial_n^i F)(a_1, \ldots, a_n) \neq 0$.
    After replacing $F$ by $F + \partial_n^{i-1}F$ if necessary, we may assume that $i=1$.
    By the implicit function theorem, there now exists a function $g$ of $n$ variables such that 
    \[
    g(a_1, \ldots, a_{n-1}, b) = a_n,
    \]
    and such that $g$ is in $\widetilde{\cF}$. 
    Hence $a_n\in D_{\widetilde{\cF}}(a_1, \ldots, a_{n-1}, b)$.
\end{proof}

\begin{cor}[cf.\ {{\cite[Lemma 2.3]{wilkie:holo}}}]
    \label{cor:D-dependent}
    Elements $w_1,\ldots,w_n\in\mathbb{C}_p$ are $LD_\mathcal{F}$-dependent if and only if there is $F\in\widetilde{\mathcal{F}}$ such that $\mathbf{w}=(w_1,\ldots,w_n)\in\mathrm{dom}(F)$, $F(\mathbf{w})=0$ and there is $\Delta\subseteq\mathrm{dom}(F)$ suitable for $F$ satisfying $\mathbf{w}\in\Delta$ and $F\upharpoonright_{\Delta}\not\equiv0$.  
\end{cor}

\begin{lem}[cf.\ {{\cite[Lemma 5.3]{wilkie:holo}}}]
\label{lem:locallydef}
    For any collection $\mathcal{F}$ of analytic functions which is closed under differentiation we have that for all $\delta\in\mathrm{Der}_{\CC_p}(\mathcal{F})$, $\delta$ respects all functions which are locally definable from $\mathcal{F}$. 
\end{lem}

\begin{proof}
    Let $F$ be locally definable from $\cF$.
    Since the statement is local, we may assume that $F$ is definable in $(\CC_p, \cF^{\PR})$.
    But $(\CC_p, \cL_{\cA_\cF})$ is an expansion by definition of $(\CC_p, \cF^{\PR})$, and so $F$ is also definable in this structure.
    By Lemma~\ref{lem:derivationspaces}, $\delta$ respects all elements of $\cA_\cF$.
    Now by~\cite[Thm.\,6.3.8]{CL} every definable function in $(\CC_p, \cL_{\cA_\cF})$ is locally given by a term in a slightly larger language from~\cite[Sec.\,6.2]{CL}, and hence $\delta$ respects every definable function in $(\CC_p, \cL_{\cA_\cF})$.
\end{proof}

\begin{defn}
Given a collection of functions $\mathcal{F}$ and a subset $X\subseteq\CC_p$, we define 
\[\mathcal{F}\mathrm{cl}(X):=\{a\in\CC_p : \forall\delta\in\mathrm{Der}_{\CC_p}(\mathcal{F})(\delta(X) = \{0\}\implies \delta(a)=0)\}.\]
\end{defn}

\begin{thm}[cf.\ {{\cite[Theorem 1.10]{wilkie:holo}}}]
\label{thm:localdef}
    Given a collection $\mathcal{F}$ of analytic functions which is closed under differentiation, the operators $D_{\widetilde{\mathcal{F}}}$ and $LD_{\mathcal{F}}$ are both pregeometries and they are identical. 
    
    Furthermore, given $D_{\widetilde{\mathcal{F}}}$-independent elements $a_1,\ldots,a_n\in \CC_p$ and a function $F$ which is locally definable from $\mathcal{F}$ with $\mathbf{a}\in\mathrm{dom}(F)$, there is $G\in\widetilde{\mathcal{F}}$ with $\mathbf{a}\in\mathrm{dom}(G)$ and there is a suitable $\Delta\subseteq\mathrm{dom}(F)\cap\mathrm{dom}(G)$, such that $G\upharpoonright_{\Delta} = F\upharpoonright_{\Delta}$. 
\end{thm}
\begin{proof}
We first observe that any function in $\widetilde{\mathcal{F}}$ is locally definable from $\mathcal{F}$, so $D_{\widetilde{\mathcal{F}}}(X)\subseteq LD_{\mathcal{F}}(X)$ for all $X\subseteq\mathbb{C}_p$. 

Next we show that $LD_{\mathcal{F}}(X)\subseteq\mathcal{F}\mathrm{cl}(X)$. 
Given $a\in LD_{\mathcal{F}}(X)$ there exist an analytic function $F$ locally definable from $\mathcal{F}$ and elements $x_1,\ldots,x_n\in X$ so that $\mathbf{x}=(x_1,\ldots,x_n)\in\mathrm{dom}(F)$ and $F(\mathbf{x}) = a$. 
By Lemma \ref{lem:locallydef} we know that any $\delta\in \mathrm{Der}_{\CC_p}(\mathcal{F})$ will respect $F$ on its domain, in particular if $\delta$ vanishes on $x_1,\ldots,x_n$, then $\delta(a)=0$.
This proves that $LD_{\mathcal{F}}(X)\subseteq\mathcal{F}\mathrm{cl}(X)$. 

To complete the proof that $D_{\widetilde{\mathcal{F}}}$ and $LD_{\mathcal{F}}$ are identical, it suffices to prove that  $\mathcal{F}\mathrm{cl}(X)\subseteq D_{\widetilde{\mathcal{F}}}(X)$.
By Lemma \ref{lem:pregeometryoperators} we know that $D_{\widetilde{\mathcal{F}}}$ is a pregeometry, so given $w\notin D_{\widetilde{\mathcal{F}}}(X)$, we may choose a $D_{\widetilde{\mathcal{F}}}$-basis of $\CC_p$ of the form $B\cup\{w\}$.
Define a map $\delta:\CC_p\to\CC_p$ in the following way. 
Given $c\in\CC_p$ there is a function $F\in\widetilde{\mathcal{F}}$ and elements $b_1,\ldots,b_n\in B$ so that $c= F(b_1,\ldots,b_n,w)$. 
Set $\delta(c):=\frac{\partial F}{\partial z_{n+1}}(b_1,\ldots,b_n,w)$.
Just as in the proof of Theorem~\ref{thm:derivations}, one uses $1$-h-minimality of $(\CC_p, \cF^\PR)$ to show that $\delta$ is well-defined and indeed a derivation.
One then gets that $\delta$ vanishes on $X$ and satisfies $\delta(w)=1$, which then proves that $\mathcal{F}\mathrm{cl}(X)\subseteq D_{\widetilde{\mathcal{F}}}(X)$ as desired. 

The `furthermore' clause now follows from Corollary \ref{cor:D-dependent}.
\end{proof}

\subsection{Applications to \texorpdfstring{$p$}{p}-adic Schanuel}
\label{subsec:p-adicschanuel}
We first recall the classical conjecture for the usual complex exponential function, see \cite{lang:introtrans}. 

\begin{conj}[Schanuel]
    Let $z_1,\ldots,z_n\in\CC$ be $\QQ$-linearly independent. 
    Then
    \begin{equation*}
        \mathrm{tr.deg.}_\QQ \QQ(z_1,\ldots,z_n,\exp(z_1),\ldots,\exp(z_n))\geq n.
    \end{equation*}
\end{conj}

One can formulate a similar conjecture for the $p$-adic exponential function, see \cite{mariaule:padicexponential,mariaule:phdthesis}.

\begin{conj}[$p$-adic Schanuel]
\label{conj:$p$-adicschanuel}
    Suppose $z_1,\ldots,z_n\in\DD_p\subset \CC_p$ are $\QQ$-linearly independent. Then
    \begin{equation*}
        \mathrm{tr.deg.}_\QQ \QQ(z_1,\ldots,z_n,\exp(z_1),\ldots,\exp(z_n))\geq n.
    \end{equation*}
\end{conj}

When used to obtain conditional results, often only a weaker form of the $p$-adic Schanuel conjecture is employed, where $z_1,\ldots,z_n$ are further assumed to be algebraic, see for example~\cite{padic.schanuel, skolem.meets.schanuel}. 

\begin{remark}
Schanuel's conjecture (both the usual one and the $p$-adic one) may be reformulated in more geometric terms in the following way:
given an algebraic variety $V\subseteq\mathbb{G}_a^n\times\mathbb{G}_m^n$ defined over a number field and with $\dim V<n$, for any point $\mathbf{x}$ in the domain of the exponential map (where the choice of domain depends on whether we are working or $\mathbb{C}$ or over $\mathbb{C}_p$) we have that if $(\mathbf{x},\exp(\mathbf{x}))\in V$, then the coordinates of $\mathbf{x}$ are $\mathbb{Q}$-linearly dependent. 
\end{remark}

Kirby and Zilber show in \cite{kirby-zilber:unifschanuel} that Schanuel's conjecture for real numbers implies a uniform version of itself (a uniform version of Schanuel's conjecture for complex numbers can only be achieved using the also open Zilber--Pink conjecture for powers of $\GG_m$, sometimes called the conjecture on intersection with tori, see \cite{kirby-zilber}). 
We show now that the same holds for $p$-adic exponentiation. 

\begin{proof}[Proof of Theorem \ref{thm:unifpschanuel}]
    A proper $\QQ$-linear subspace of $\GG_a^n$ is determined by equations of the form
    \[m_1X_1 + \cdots + m_nX_n = 0\]
    where $m_1,\ldots,m_n\in\ZZ$ are not all zero.
    To prove the theorem, we will show that Conjecture \ref{conj:$p$-adicschanuel} implies that there is a positive integer $N$ (depending only on $V$ and possibly $p$) so that the set $X:= \{\mathbf{x}\in\CC_p^n : \mathbf{x}\in\DD_p^n \wedge (\mathbf{x},\exp(\mathbf{x})) \in V\}$ is contained in the union of proper $\QQ$-linear subspaces whose coefficients are bounded above by $N$ in absolute value. 
    
    First observe that the set $X$ is definable a 1-h-minimal expansion of $\CC_p$. 
    By Conjecture \ref{conj:$p$-adicschanuel}, for every $\mathbf{x}\in X$ there are $m_1,\ldots,m_n\in\mathbb{Z}$ such that $m_1x_1+\cdots+m_nx_n=0$. 
    If the 1-h-minimal dimension of $X$ is zero, then $X$ is a finite set, so in this case the bound $N$ can be obtained by taking the maximum of the $|m_i|$ appearing in the $\mathbb{Z}$-linear combinations of each $\mathbf{x}\in X$. 

    From now on we assume that $X$ has positive 1-h-minimal dimension. 
    Now~\cite[Thm.\,6.3.8]{CL} shows that we have analytic cell decomposition, and hence there is a positive dimensional cell $C\subseteq X$ and a definable analytic diffeomorphism $\theta: B\to C$ from a product of open disks and annuli $B$ to $C$. 
    Given $a,b\in C$ distinct let $x,y\in B$ be so that $\theta(x)=a$ and $\theta(y)=b$.
    There exists some annulus $A\subset \cO_{\CC_p}$ containing $0$ and $1$ such that the map $\gamma: A\to B: t\mapsto t\mathbf{x} + (1-t)\mathbf{y}$ is well-defined.
    Since there are only countably many $\mathbb{Z}$-linear relations and uncountably many elements in $A$, by Conjecture \ref{conj:$p$-adicschanuel} there must be some equation of the form $m_1x_1+\cdots +m_nx_n=0$, with $m_1,\ldots,m_n\in\mathbb{Z}$, which is satisfied by $\gamma(t)$ for uncountably many $t\in A$. 
    By 1-h-minimality, the set $\{t\in A : m_1\gamma(t)_1+\cdots +m_n\gamma(t)_n=0\}$ must contain an open ball. 
    By the identity principle, this implies that $m_1\gamma(t)_1+\cdots +m_n\gamma(t)_n=0$ for all $t\in A$. 
    We can now proceed inductively in the dimension on the cell to conclude that every point in $C$ must satisfy the same $\mathbb{Z}$-linear relation. 
\end{proof}

Next, we wish to prove some weak forms of Conjecture \ref{conj:$p$-adicschanuel}.
For this we recall the theorem of Ax for exponentiation in the setting of differential fields.

Let $F$ be a field and $C\subset F$ a subfield. 
Recall that given a field $F$ of characteristic zero, and a subfield $C\subset F$, elements $x_1, \ldots, x_n\in F$ are said to be \emph{$\QQ$-linearly independent modulo $C$} if $C$ does not contain any non-trivial $\QQ$-linear combinations of the $x_i$.

\begin{thm}[{{\cite[Theorem 3]{ax}}}]
\label{thm:ax}
    Let $(F,\Delta)$ be a differential field of characteristic zero with field of constants $C$. 
    Suppose that $x_1, \ldots, x_n, y_1, \ldots, y_n\in F$ are such that $\partial(y_i) = y_i\partial(x_i)$ for all $i\in\{1,\ldots,n\}$ and all $\partial\in\Delta$, and $x_1,\ldots,x_n$ are $\QQ$-linearly independent modulo $C$. Then
    \begin{equation*}
        \mathrm{tr.deg.}_C C(x_1, \ldots, x_n, y_1, \ldots, y_n)\geq n + \rk(\partial x_i).
    \end{equation*}
\end{thm}

As we mentioned before, the work of Kirby in \cite{kirby-expalg} (along with fixes provided in \cite{aek-differentialEC,aek-closureoperator}) for abstract exponential fields uses Theorem \ref{thm:ax} to deduce the existence of non-trivial derivations respecting exponentiation in the case of $(\CC,\exp)$ and $(\RR,\exp)$.
Although not mentioned in \cite{kirby-expalg}, the same argument would give the existence of such derivations in $(\CC_p,\exp)$. 
From this one gets the following two corollaries, both of which can be seen as weak forms of Conjecture \ref{conj:$p$-adicschanuel}. 

\begin{cor}[cf {{\cite[Theorem 1.4]{kirby-expalg}}}]
\label{cor:countable.subfield.schanuel}
    There exists a countable subfield $C$ of $\CC_p$ such that if $z_1, \ldots, z_n\in \DD_p$ are $\QQ$-linearly independent modulo $C$, then
    \[
    \mathrm{tr.deg.}_C C(z_1,\ldots,z_n,\exp(z_1),\ldots,\exp(z_n))> n.
    \]
\end{cor}

\begin{proof}
Consider $\CC_p$ as an $\cL$-structure, where $\cL$ is the countable language from Corollary~\ref{cor:countable.language.Cpexp}. 
Let $\Delta=(\partial_\tau)_{\tau\in \cB}$ be the collection of derivations on $\CC_p$ constructed using Theorem~\ref{thm:derivations}. 
We take $C = \bigcap_\tau \ker \partial_\tau$, which is indeed a countable subfield by Theorem~\ref{thm:derivations}(\ref{it:kernel.derivations}). 
If $z_1, \ldots, z_n\in \DD_p$ are  $\QQ$-linearly independent modulo $C$, we conclude via Theorem \ref{thm:ax}.
\end{proof}

\begin{defn}
    Let $\Delta_{\exp}$ denote the collection of all field derivations on $\CC_p$ which respect $\exp$, and define
    \[C_{\exp}:=\bigcap_{\partial\in\Delta_{\exp}}\ker\partial.\]
    As we saw in the proof of Corollary \ref{cor:countable.subfield.schanuel}, $C_{\exp}$ is a countable set. 
    An element $\lambda\in\mathbb{C}_p$ is called \emph{exponentially transcendental} if $\lambda\notin C_{\exp}$. 
\end{defn}

A somewhat more explicit description of the elements of $C_{\exp}$ can be given as solutions to \emph{Khovanskii systems}, see \cite[Theorem 1.1]{kirb-weierstrass}. 

By \cite[Theorem 1.2]{bays-kirby-wilkie} we also get the following version of Conjecture \ref{conj:$p$-adicschanuel} for exponentially transcendental powers.

\begin{cor}
    Suppose $\lambda\in\DD_p$ is exponentially transcendental. 
    Then for all $x_1,\ldots,x_n\in\DD_p$ which are $\QQ$-linearly independent we have
    \[\mathrm{tr.deg.}_{\QQ(\lambda)}\QQ(\lambda, \exp(x_1),\ldots,\exp(x_n),\exp(\lambda x_1),\ldots,\exp(\lambda x_n))\geq n.\]
\end{cor}

\section{Blurring Tate Uniformization}
\label{sec:blurredtate}

In the rest of the paper we will shift our attention from $p$-adic exponentiation to Tate's uniformization map for elliptic curves. 

Throughout this section we fix some $q\in \CC_p, v(q) > 0$ and work with the elliptic curve $E = E_q$ with uniformization map $\phi = \phi_q$.
We begin by showing a differential Ax--Schanuel result for Tate's uniformization map. 
Recall that this map satisfies the differential equation (\ref{eq:diffeqtate}).
The result will follow from a mixed Ax--Schanuel theorem which combines elliptic exponential maps with the usual exponential, and which may be found in work of Ax \cite{ax2} and Kirby \cite{kirb-weierstrass}.

\begin{prop}[Ax--Schanuel for Tate's uniformization map]\label{prop:ax.schanuel.tate}
    Let $(K,\Delta)$ be a differential field of characteristic zero where $\Delta$ is a set of $m$ commuting derivations, and let $C$ denote its field of constants.
    Let $f\in K[t]$ be a cubic polynomial without repeated roots and such that the elliptic curve defined by $y^2 = f(t)$ does not have complex multiplication. 
    Suppose $u_1,\ldots,u_n,y_1,\ldots,y_n\in K^\times$ are such that
    $(u_i\partial y_i)^2 = f(y_i)(\partial u_i)^2$ for all $\partial\in\Delta$ and all $i\in\{1,\ldots,n\}$. 
    If $u_1,\ldots,u_n$ are multiplicatively independent modulo $C$, then 
    \[\mathrm{tr.deg.}_CC(u_1,\ldots ,u_n,y_1,\ldots,y_n)\geq n+\mathrm{rank}(\partial u_i)_{i\in\{1,\ldots,n\}, \partial\in\Delta}.\]
\end{prop}
\begin{proof}
    By passing to a differential field extension if necessary, we let $x_1,\ldots,x_n\in K^\times$ be such that $\frac{\partial u_i}{u_i}=\partial x_i$ for all $\partial\in \Delta$ and all $i\in\{1,\ldots,n\}$. 
    Then the equation $(u_i\partial y_i)^2 = f(y_i)(\partial u_i)^2$ implies that $(\partial y_i)^2 = f(y_i)(\partial x_i)^2$.
    By \cite[Proposition 3.2]{kirb-weierstrass} (see also \cite{kirby-weierstrasscorrigendum}) we have that
    \[\mathrm{tr.deg.}_CC(x_1,\ldots,x_n,u_1,\ldots,u_n,y_1,\ldots,y_n)\geq 2n+\mathrm{rank}(\partial x_i)_{i\in\{1,\ldots,n\}, \partial\in\Delta}\]
    unless $x_1,\ldots,x_n$ are $\QQ$-linearly dependent modulo $C$, which is equivalent to saying that $u_1,\ldots,u_n$ are multiplicatively dependent modulo $C$. 
    
    Finally, we observe that any $a_1,\ldots,a_n\in K$ we have
    \[a_1\begin{pmatrix}
        \partial_1 x_1\\
        \vdots\\
        \partial_m x_1
    \end{pmatrix} +\cdots + a_n\begin{pmatrix}
        \partial_1 x_n\\
        \vdots\\
        \partial_m x_n
    \end{pmatrix} = 0\implies \frac{a_1}{u_1}\begin{pmatrix}
        \partial_1 u_1\\
        \vdots\\
        \partial_m u_1
    \end{pmatrix} + \cdots + \frac{a_n}{u_n}\begin{pmatrix}
        \partial_1 u_n\\
        \vdots\\
        \partial_m u_n
    \end{pmatrix} = 0,\]
    showing that $\mathrm{rank}(\partial u_i)_{i\in\{1,\ldots,n\}, \partial\in\Delta}\leq \mathrm{rank}(\partial x_i)_{i\in\{1,\ldots,n\}, \partial\in\Delta}$, which completes the proof. 
\end{proof}

There is a natural action of $\MM_n(\ZZ)$ (the set of $n\times n$ matrices with coefficients in $\ZZ$) on $\GG_m^n$.
Indeed, denote the entries of $M=(m_{ij})_{i,j}$, and let  $\mathbf{w}=(w_{1},\ldots,w_{n})$ be any element of $\mathbb{G}_m^n$. 
We define the action in the following way: 
\begin{equation*}
\mathbf{w}^M:=\left(\prod_{j=1}^{n}w_j^{m_{1,j}},\ldots,\prod_{j=1}^{n}w_j^{m_{n,j}}\right).
\end{equation*}
We also have an action of $\MM_n(\ZZ)$ on $E^n$ given by, 
\[M\mathbf{x}:=\left(\sum_{j=1}^nm_{1,j}x_{j},\ldots,\sum_{j=1}^nm_{n,j}x_{j}\right),\]
and therefore combining both definitions, $\MM_n(\ZZ)$ also acts on $\GG_m^n\times E^n$ by setting $M(\mathbf{x},\mathbf{y}):=\left(\mathbf{x}^M, M\mathbf{y}\right)$.

\begin{defn}
    Let $V\subset \GG_m^n\times E^n$ be an irreducible algebraic variety.
    We say that $V$ is \emph{rotund} if for every $M\in \MM_n(\ZZ)$ we have $\dim(M V)\geq \rk(M)$.
\end{defn}

\begin{lem}\label{lem:V_M}
    Let $V\subset \GG_m^n\times E^n$ be an irreducible rotund algebraic variety, and let $M\in \MM_n(\ZZ)$ have rank $n-r$ for some $r\geq 1$.
    Let $T\subset \GG_m^n\times E^n$ be the algebraic subgroup given by all $(X,Y)\in \GG_m^n\times E^n$ for which $X^M = 1$ and $MY = 0$.
    Then there exists a non-empty Zariski-open subset $V_M\subset V$ such that for $\gamma\in V_M$ we have 
    \[
    \dim (V\cap \gamma + T) \leq \dim V - n + r.
    \]
\end{lem}
\begin{proof}
    Consider the regular map
    \[
    F: V\to \GG_m^n\times E^n: (X,Y)\mapsto (X^M, MY).
    \]
    The fibres of $F$ are of the form $V\cap \gamma + T$. 
    By the fibre-dimension theorem, there exists a non-empty Zariski open subset $V_M$ of $V$ such that for $\gamma \in V_M$ we have
    \[
    \dim(V\cap \gamma + T) = \dim(V)-\dim(MV).
    \]
    By rotundity, $\dim(MV)\geq \rk M = n-r$, and hence $\dim(V\cap \gamma+T) \leq \dim(V) - n + r$.
\end{proof}


\begin{proof}[Proof of Theorem~\ref{thm:tateblurring}]
    Write $m = \dim V$ and note that since $V$ is rotund we have that $m\geq n$.
    Consider the map
    \[
    \theta: V\to E^n: (\mathbf{x},\mathbf{y})\mapsto \mathbf{y}-\phi(\mathbf{x}).
    \]
    We will show that most fibres of $\theta$ are locally of dimension at most $m-n$, from which the result will follow.

    Let $\cL_0$ be the countable language from Corollary~\ref{cor:countable.language.Cpexp} expanded by a symbol for $q$.
    This gives a $1$-h-minimal structure on $\CC_p$.
    
    Take $\mathbf{v}\in V$.
    Then there is an open definable neighbourhood $V'\subset V$ of $\mathbf{v}$ in $V$ such that the restriction of $\theta$ to $V'$ is $\cL_0$-definable.
    Note that the local dimension of $V'$ at every point is equal to $m$.
    Put $t = \theta(\mathbf{v})$ and $A = \theta^{-1}(t)\cap V'$, which is an $\cL$-definable subset of $V$.
    Put $\cL = \cL_0(t)$.

    Assume that $\dim A > m-n$. 
    By~\cite[Lem.\,5.3.5]{CHR} dimension in h-minimal structures coincides with $\acl$-dimension, and so there exists a point $(x_1, \ldots, x_n, y_1, \ldots, y_n)\in A$ whose $\acl$-dimension is at least $m-n+1$.
    Since $\phi$ is $\cL$-definable, we find $1\leq i_1 < \ldots < i_{m-n+1}\leq n$ such that the elements $x_{i_1}, \ldots, x_{i_{m-n+1}}$ are $\acl$-independent.
    The proof of Theorem~\ref{thm:derivations} gives $m-n+1$ commuting derivations $\partial_1, \ldots, \partial_{m-n+1}$ with the following properties:
    \begin{enumerate}
        \item $\partial_j x_{i_\ell} = 1$ if $j=\ell$ and $0$ else,
        \item $\partial_j t_i = 0$ for every $i,j$, and
        \item $\partial_j$ respects $\phi$.
    \end{enumerate}
    
    Write $y_i = (y_{i1}, y_{i2}), t_i = (t_{i1}, t_{i2})$ and $\phi = (\phi_1, \phi_2)$.
    Then we have that $y_{i1} = t_{i1} + \phi_1(x_i)$ and hence $\partial_j y_{i1} = \partial_j (\phi_1(x_i))$.
    Since $\partial_j$ respects $\phi$, we therefore obtain that
    \[
    (x_i \partial_j (y_{i1}-t_{i1}))^2 = f(y_{i1}-t_{i1})(\partial_j x_i)^2,
    \]
    where $f(z) = 4z^3 + z^2 + 4a_4(q)z + 4a_6(q)$ is the polynomial appearing in the differential equation (\ref{eq:diffeqtate}) for $\phi_1$.

    Let $C = \cap_j \ker \partial_j$.
    Since $\partial_j$ respects $\phi$, we see that $C$ is closed under taking $\phi$.
    Let $k$ be the smallest integer for which there exists a matrix $M\in \MM_n(\ZZ)$ of rank $n-k$ such that
    \[
    \gamma = (\mathbf{x}^M, M\mathbf{y}) \in \GG_m^n(C)\times E^n(C).
    \]
    Let $T\leq \GG_m^n\times E^n$ be the subgroup defined by $\mathbf{X}^M = 1$ and $M\mathbf{Y} = 0$, so that $(\mathbf{x},\mathbf{y}) = \gamma\in T$.
    Consider $F = C(\mathbf{x},\mathbf{y})$, so that $(\mathbf{x},\mathbf{y})$ is an $F$-point on the algebraic variety $V\cap (\gamma+T)$ which is defined over $C$. 
    Hence
    \[
    \trdeg_C C(\mathbf{x},\mathbf{y})\leq \dim(V\cap (\gamma+T)).
    \]
    By Proposition~\ref{prop:ax.schanuel.tate} we have $\trdeg_C C(\mathbf{x},\mathbf{y}) -k\geq m-n+1 = \rk(\partial_j x_{i_\ell})_{j,\ell}$ and hence
    \[
    \dim(V\cap (\gamma+T)) > k + m-n.
    \]

    For $M\in \MM_n(\ZZ)$ of rank at most $n-1$ let $V_M\subset V$ be the non-empty Zariski-open subset of $V$ from Lemma~\ref{lem:V_M}.
    Consider the set
    \[
    W = V'\cap  \bigcap_{M} V_M(\CC_p)
    \]
    where the final intersection is over all $M\in \MM_n(\ZZ)$ of rank at most $n-1$.
    Note that this is the intersection of $V'$ with countably many non-empty Zariski-open subsets of $V$. 
    Since $\CC_p$ is uncountable, $W$ is therefore dense in $V'$ in the $p$-adic topology.
    Let $V_0\subset V$ consist of all $v\in V'$ for which $\dim(V'\cap \theta^{-1}(\theta(v))) \leq m-n$, which is a definable subset of $V'$.
    By the arguments above, $W$ is contained in $V_0$.
    We claim that for every open subset $B\subset V'$ we have $\dim V_0\cap B = m$.
    To see this, note that $W\cap B$ is dense in $B$, and hence since $W\subset V_0$, also $V_0\cap B$ is dense in $B$.
    But both $B$ and $V_0\cap B$ are definable and $\dim B = m$, so that also $\dim V_0\cap B = m$.

    Now, dimension theory implies that $\theta(B\cap V_0)\subset E^n$ has dimension $n$ and so has non-empty interior.
    Therefore, $\theta(B\cap V_0)$ contains an elements with coordinates from $Q$, say $\mathbf{r}$.
    This means that for some $(\mathbf{x},\mathbf{y})\in B$ we have
    \[
    \theta(\mathbf{x},\mathbf{y})=\mathbf{y}-\phi(\mathbf{x}) = \mathbf{r}. \qedhere
    \]
\end{proof}

\subsection{Quasiminimality}
\label{sec:quasiminimality}

\begin{defn}
    A first-order structure $M$ is said to be \emph{quasiminimal} if for every countable subset $A\subseteq M$ and every subset $S\subseteq M$, if $S$ is invariant under the action of $\mathrm{Aut}(M/A)$, then $S$ is countable or its complement in $M$ is countable. 
\end{defn}

Given a subset $Q\subseteq E$, define the set
\[T_Q:=\{(x,y)\in \CC_p^\times\times E : \exists b\in Q : y = \phi( x +  b)\}.\]
We call $T_Q$ the \emph{blurring of $\phi$ by $Q$}. 

Blurrings have previously been considered for other functions. 
In \cite{kirby2019blurred}, Kirby shows that some blurrings of the complex exponential define quasiminimal structures on $\CC$. 
An analogous result for blurrings of the modular $j$-function are obtained in \cite[\S 5]{aslanyan-kirby}.
In our setting, it also makes sense to ask whether $(\CC_p, +, \cdot, T_Q)$ is quasiminimal. 
It is important to note that this structure does not include the valuation, one should think about it as the expansion of the pure field structure on $\CC_p$ by the relation $T_Q$.

Whether $(\CC_p, +, \cdot, T_Q)$ is quasiminimal depends on the choice of $Q$.
The details behind the discussion below would take us far away from any matters related to 1-h-minimality, so we have not pursued them here, but we give the following overview for the interested reader. 
As explained in \cite[Corollary 11.7]{bays-kirby} and \cite[\S5]{aslanyan-kirby}, to show $(\CC_p, +, \cdot, T_Q)$ is quasiminimal, two conditions need to be verified: the first is referred to as \emph{$\Gamma$-closedness} in \cite{bays-kirby} and is guaranteed by Theorem \ref{thm:tateblurring} when $Q$ is $p$-adically dense in $E^n$, and the second one which is called the \emph{countable closure property}. 
Regarding this second condition, one can in fact show something stronger, namely that the structure $(\CC_p, +, \cdot, \phi)$ has the countable closure property, for example by replicating the proof presented in \cite[Lemma 5.12]{zilberexp} for exponentiation. 
But for our purposes, the results presented in \cite[\S10D]{bays-kirby} already show that, in our case, this property follows from a differential form of Ax--Schanuel, which we have in Proposition \ref{prop:ax.schanuel.tate}.  

The results in \cite{aslanyan-kirby} and \cite{kirby2019blurred} also show that if one chooses the blurring carefully, then not only is the resulting structure quasiminimal, but one is even able to axiomatize the first-order theory of the blurred structure by using Hrushovski--Fra\"iss\'e amalgamation constructions.
Informed by these results, we expect that if one chooses $Q$ to be intersection of the kernels of all derivations on $\CC_p$ which respect $\phi$ (recall that by Theorem \ref{thm:derivations} non-trivial derivations respecting $\phi$ exist, and the intersection of their kernels is countable) then a similar result can be obtained for the first-order theory of $(\CC_p, +, \cdot, T_Q)$.

\section{Likely Intersections}\label{sec:likely}

In this section, we will give a $p$-adic analog of a result about Euclidean density of likely intersections given in \cite{eterovic-scanlon} in the context of elliptic curves admitting a Tate uniformization. 

\subsection{A dimension estimate}
Call a subset $X\subset \CC_p^n$ \emph{analytic} if for every $x\in X$ there is a closed box $B$ around $x$, such that on $B$, $X$ is defined by the vanishing of finitely many strictly convergent power series.
We will need a lower bound on the dimension of an intersection of two analytic subsets of $\CC_p^n$. 
For notions coming from rigid analytic geometry, such as affinoid and $\mathrm{Spa}(A)$, we refer the reader to \cite{fresnel-vanderput:rigid,huber:contval}.

If $X\subset \cO_{\CC_p}^n$ is an affinoid, then $X(\CC_p)\subset \cO_{\CC_p}^n$ is a definable set in the language from Section~\ref{sec:weierstrass} over the full Weierstrass system.
In that case, the dimension of $X$ as a rigid analytic space coincides with the 1-h-minimal dimension of $X(\CC_p)$.

\begin{prop}\label{prop:dim.intersection}
    Let $X,Y$ be analytic subsets of $\CC_p^n$ and let $x\in X\cap Y$.
    Then
    \[
    \dim_x(X\cap Y)\geq \dim X + \dim Y-n.
    \]
\end{prop}

\begin{proof}
    The statement is local, so we may assume that $X$ and $Y$ are the $\CC_p$-points of irreducible and reduced affinoids.
    The dimension of $X\cap Y$ is the same as the dimension of $X\times Y$ intersected with the diagonal in $\CC_p^n\times \CC_p^n$. 
    The diagonal is cut out by hyperplanes, so it suffices to check that the dimension of the irreducible components of the intersection of $Z=X\times Y$ with a linear hyperplane have dimension at least $\dim Z-1$.
    Now $Z$ consists of the $\CC_p$-points of a reduced and irreducible affinoid $\mathrm{Spa}(A)$, where $A$ is a domain. 
    By Krull's principal ideal theorem \cite[Theorem 10.2]{eisenbud:commutativealgebra}, given $f\in A$ non-zero, the irreducible components of the zero set of $f$ in $Z$ have dimension $\dim Z-1$.
\end{proof}

 \subsection{Special and weakly special varieties} 
Since $\phi_q:\mathbb{G}_m^n\to E_q^n$ is a transcendental map, it is {\bf not} usually the case that if $W\subseteq\mathbb{G}_m^n$ is an algebraic variety, then $\phi_q(W)$ is also an algebraic variety. 
But this does sometimes happen, in which case we say that $W$ is \emph{bi-algebraic} (with respect to $\phi_q$). 

\begin{defn}
    A subvariety $S\subseteq E_q^n$ is said to be \emph{weakly special} if $S$ is an irreducible component of a coset of an algebraic subgroup of $E_q^n$.
    
    We say that $S$ is \emph{special} if it is an irreducible component of a torsion coset of an algebraic subgroup of $E_q^n$.
\end{defn}

\begin{prop}
\label{prop:weaklyspecial}
    A subvariety $S\subseteq E_q^n$ is weakly special if and only if there is a variety $W\subseteq\mathbb{G}_m^n$ that is bi-algebraic with respect to $\phi_q$ and such that $\phi_q(W)=S$. 
\end{prop}

\begin{proof}
Suppose first that $S$ is bi-algebraic with respect to $\phi_q$.  
We observe that the result is immediate when $\dim S = 0$ as every point is a coset of the identity.   
Now we proceed by induction on $n$.
When $n=1$, we have that $S$ is a point (so we are done) or $S = E_q$ (so we are also done).

Now we treat the general case. 
    By the definition of bi-algebraic, there is an algebraic variety $W\subseteq\mathbb{G}_m^n$ such that $\phi_{q}(W) = S$, which implies in particular that $\dim W = \dim S$. 
    Let $\mathcal{G}_{\phi_q}^n$ denote the $n$-folds graph of $\phi_q$, that is
    \[\mathcal{G}_{\phi_q}^n:=\{(\mathbf{x},\phi_q(\mathbf{x}))\in\mathbb{G}_m^n\times E_q^n\}.\]
    Then the intersection $(W\times S)\cap\mathcal{G}_{\phi_{q}}^n$ contains the analytic subset $U=\{(\mathbf{w},\phi_{q}(\mathbf{w})) : \mathbf{w}\in W\}$ of dimension $\dim U =\dim W = \dim S$. 
    We can choose algebraic functions $f_1,\ldots,f$ so that $W$ is locally parametrized by $(x_1,\ldots,x_d,f_1(x_1,\ldots,x_d),\ldots,f_{n-d}(x_1,\ldots,x_d))$, where $d=\dim W$. 
    So, if we work in the differential field $(K,\Delta)$ of analytic functions defined on a $p$-adically open subset of $W$ (where the set $\Delta$ consists of the $d$ derivations given by partial differentiation) and we choose the above parametrization as the tuple $(u_1,\ldots,u_n)$ of elements of $K$, then we get that
    \[\dim U = \mathrm{rank}(\partial u_i)_{i\in\{1,\ldots,n\},\partial\in\Delta}.\]
    Since either $\dim S=n$ (and then $S = E_q^n$, in which case we are done) or $\dim U > \dim W + \dim S - n$, in the latter case we use Proposition \ref{prop:ax.schanuel.tate} to conclude that $\phi_q(U) = S$ is contained in a coset of a proper algebraic subgroup. 

    Suppose there are $\mathbf{c}\in E_q^n$ and a proper irreducible algebraic subgroup $G\subset E_q^n$ such that $S\subseteq \mathbf{c} +G$.
    Choose a subtorus $T\subseteq\mathbb{G}_m^n$ such that $\phi_{q}(T)=G$, and choose $\mathbf{a}\in\mathbb{G}_m^n$ such that $\phi_{q}(\mathbf{a}) = \mathbf{c}$ and $W\subseteq \mathbf{a}T$. 
    Then $-\mathbf{c}+S \subseteq G$ and $G\cong E_q^k$ for some $0\leq k<n$. 
    Furthermore, the restriction of $\phi_q$ to $T$ shows that $-\mathbf{c} + S$ is bi-algebraic (as witnessed by $\mathbf{a}^{-1}W$).
    Since $\dim G<n$, by the induction hypothesis $-\mathbf{c}+S$ is a coset of an algebraic subgroup, from which we conclude that $S$ is also a coset of an algebraic subgroup. 

Conversely, suppose $S = \mathbf{c}+G$, for some algebraic subgroup $G\subseteq E_q^n$. 
    Then there is a subtorus $T\subseteq\mathbb{G}_m^n$ such that $\phi_{q}(T) = G$. 
    Choosing $\mathbf{a}\in\mathbb{G}_m^n$ such that $\phi_{q}(\mathbf{a}) = \mathbf{c}$ shows that $S = \phi_{q}(\mathbf{a}T)$, and hence $S$ is bi-algebraic. 
\end{proof}


\subsection{Atypical intersections}
In order to motivate Theorem \ref{thm:likely}, we first introduce the Zilber--Pink conjecture for powers of the elliptic curve $E_q$. 

\begin{defn}
Suppose that $V$ and $W$ are subvarieties of an irreducible smooth algebraic variety $Z$. 
We say that $V$ and $W$ are \emph{likely to intersect}, if
\begin{equation*}
    \dim V+\dim W\geq \dim Z. 
\end{equation*}
We say that $V$ and $W$ are \emph{unlikely to intersect} if the opposite happens
\begin{equation*}
    \dim V+ \dim W<\dim Z.
\end{equation*}
Now assume that $V\cap W\neq\emptyset$, and let $X$ be an irreducible component of the intersection $V\cap W$.
We say that $X$ is an \emph{atypical component of $V\cap W$ (in $Z$)} if
\begin{equation*}
    \dim X > \dim V + \dim W - \dim Z.
\end{equation*}

We say that the intersection \emph{$V\cap W$ is atypical (in $Z$)} if it has an atypical component. 
\end{defn}

\begin{defn}
    Given a subvariety $V\subseteq E_q^n$, there is a smallest special subvariety containing $V$ which we call the \emph{special closure} of $V$, and we denote this by $\mathrm{spcl}(V)$. 
    We define the \emph{defect} of $V$ as
    \[\mathrm{def}(V) := \dim \mathrm{spcl}(V) - \dim V.\]
\end{defn}

Given a positive integer $n$ and $k\in\{0,\ldots,n\}$, let $\mathscr{S}^k$ denote the set of special subvarieties of $E_q^n$ of dimension at most $k$. 
The Zilber--Pink conjecture can be formulated in different ways, the version we give below is phrased in terms of unlikely intersections.

\begin{conj}[Zilber--Pink, unlikely version]
    \label{conj:zpunlikely}
    For every positive integer $n$ and any algebraic variety $V\subseteq E_q^{n}$, the set
\[\bigcup_{S\in\mathscr{S}^{\mathrm{def}(V)-1}}V\cap S\]
is not Zariski dense in $V$. 
\end{conj}

One can also formulate the conjecture in terms of atypical components.

\begin{defn}
We say that $X$ is an \emph{atypical component} of the variety $V\subseteq E_q^n$ if there exists a special subvariety $S$ of $E_q^n$ such that $X$ is an atypical component of $V\cap S$. 
\end{defn}

\begin{conj}[Zilber--Pink, atypical version]
    For every positive integer $n$, every subvariety $V\subset E_q^n$ contains only finitely many maximal (with respect to inclusion) atypical components. 
\end{conj}

We refer the reader to Pila's book \cite{pila_2022} for the equivalence of the two statements, as well as yet another formulation of the conjecture in terms of so-called \emph{optimal subvarieties}. 

Although Conjecture \ref{conj:zpunlikely} remains open, an important result towards it is the following theorem which shows that the atypical components of a variety $V \subseteq E_q^n$ are contained in the cosets of finitely many algebraic subgroups, or in other words, the atypical components of $V$ are contained in finitely many families of weakly special varieties. 

\begin{thm}[Weak Zilber--Pink, {{\cite[Theorem 4.6]{Kirby-semiab}}}]
    \label{thm:weakzp}
Let $K$ denote an algebraically closed field of characteristic zero so that $E_q$ is definable over $K$. 
For every positive integer $n$ and for every $K$-definable family $\{V_b\}_{b\in B}$ of subvarieties of $E_q^n$, there exists a finite collection $\mathscr{H}$ of proper irreducible algebraic subgroups of $E_q^n$ with the following property: for every connected algebraic subgroup $T$ of $E_q^n$, every $\mathbf{c}\in E_q^n\left(K\right)$ and every $b\in B(K)$, if $X$ is an atypical component of $V_b\cap(\mathbf{c}+T)$, then there is $H\in\mathscr{H}$ such that $X\subseteq(\mathbf{c}'+H)$ for some $\mathbf{c'}\in\mathbf{c}+T(K)$, $\dim H\leq \dim V_b+ \dim T - \dim X$ and
\begin{equation*}
    \dim X\leq \dim V_b\cap(\mathbf{c}'+H) + \dim T\cap H - \dim H.
\end{equation*}
Furthermore, if $\mathbf{c}+T$ is the weakly special closure of $X$, then $T\subseteq H$. 
\end{thm}

\subsection{Persistently likely intersections}
\label{subsec:persistent}
\begin{defn}
Let $V$ and $W$ be subvarieties of $E_q^n$. 
We say that the intersection between $V$ and $W$ is \emph{persistently likely} if for every algebraic subgroup $T\subseteq E_q^n$ we have $\dim \psi(V) + \dim \psi(W)\geq n-\dim T$, where $\psi:E_q^n\to E_q^n/T$ denotes the quotient map.  
\end{defn}

The following elementary lemma shows that if $V$ and $S$ are persistently likely to intersect, where $S$ is weakly special, then so are $V$ and $\mathbf{c}+S$ for all $\mathbf{c}\in E_q^n$.

\begin{lem}
\label{lem:dimcosetintersection}
    Let $T,S$ be two irreducible algebraic subgroups of $E_q^n$. 
    For every $\mathbf{a}\in S$, every $\mathbf{c}\in E_q^n$ and every $\mathbf{b}\in \mathbf{c}+S$ we have
    \[\dim S\cap (\mathbf{a}+T) = \dim(\mathbf{c}+S)\cap(\mathbf{b}+T).\]
\end{lem}
\begin{proof}
    Define $\varphi(\mathbf{x}):=\mathbf{b}-\mathbf{a}+\mathbf{x}$. 
    Observe that given $\mathbf{s}\in S$, $\varphi(\mathbf{s}) = \mathbf{b}-\mathbf{a}+\mathbf{s} + \mathbf{c} - \mathbf{c} \in \mathbf{c} + S$, and given $\mathbf{t}\in T$, $\varphi(\mathbf{a}+\mathbf{t}) = \mathbf{b}+\mathbf{t}\in\mathbf{b}+T$.
    Therefore, $\varphi$ defines a bijection between $S\cap (\mathbf{a}+T)$ and $(\mathbf{c}+S)\cap (\mathbf{b}+T)$, which finishes the proof.  
\end{proof}

\begin{cor}
\label{cor:perslikelycosets}
Let $T,S$ be two irreducible algebraic subgroups of $E_q^n$ and let $\psi:E_q^n\to E_q^n/T$ denote the quotient map. 
Then for every $\mathbf{c}\in E_q^n$ we have $\dim\psi(S) = \dim\psi(\mathbf{c}+S)$.
\end{cor}
\begin{proof}
    Choose a finitely generated field $K\subseteq\CC_p$ over which $E_q$, $S$, and $T$ are defined. 
    Choose $\mathbf{a}\in S$ generic over $K$. 
    Then by the fiber-dimension theorem we have $\dim S = \dim (\mathbf{a}+T)\cap S+\dim\psi(S)$. 

    Choose $\mathbf{c}\in E_q^n$ and let $\mathbf{b}\in \mathbf{c}+S$ be generic over $K(\mathbf{c})$. 
    As above, we get $\dim(\mathbf{c}+S) = \dim(\mathbf{c}+S)\cap(\mathbf{b}+T) + \dim\psi(\mathbf{c}+S)$. 
    Since $\dim S = \dim(\mathbf{c}+S)$, the result follows from Lemma \ref{lem:dimcosetintersection}. 
\end{proof}

Although the definition of persistent likely intersections considers all possible quotients of $E_q^n$, the next lemma shows that by weak Zilber--Pink, it suffices to consider an appropriate finite collection of quotient maps. 
Of course, trying to determine this finite collection for a given variety $V$ could only be done in practice if one had an effective weak Zilber--Pink statement.

\begin{lem}
    \label{lem:persislikelyfinite}
    Let $V\subseteq E_q^n$ be an irreducible subvariety. 
    There exist finitely many proper algebraic subgroups $H_1,\ldots,H_n$ of $E_q^n$ such that for every special subvariety $S\subset E_q^n$ we have: $V$ and $S$ are persistently likely to intersect if and only if for every $i\in\{1,\ldots,n\}$ we have $\dim\psi_i(V) + \dim\psi_i(S)\geq n-\dim H_i$, where $\psi_i:E_q^n\to E_q^n/H_i$ denotes the quotient map.
\end{lem}
\begin{proof}
One direction is obvious, so we will focus on the other one.
Let $\mathscr{H}$ be the finite collection of proper subgroups of $E_q^n$ satisfying the conclusion of Theorem \ref{thm:weakzp} applied to the constant family $V_p=V$. 
    Suppose $T$ is a proper irreducible algebraic subgroup of $E_q^n$ such that $\dim\psi(V) + \dim\psi(S) < n-\dim T$, where $\psi:E_q^n\to E_q^n/T$ denotes the quotient map.
    Let $S_0$ be an irreducible algebraic subgroup of $E_q^n$ and let $\mathbf{c}\in\mathrm{Tor}(E_q^n)$ be such that $S = \mathbf{c} + S_0$. 
    Let $K\subset\CC_p$ be a finitely generated subfield such that $E_q$, $V$, $S_0$, $T$ and every element of $\mathscr{H}$ are defined over $K$. 
    Observe that then every element of $\mathrm{Tor}(E_q^n)$ is defined over $\overline{K}$. 

    Consider the quotient map $\psi':E_q^n\to E_q^n/(S_0+T)$. 
    Choose $\mathbf{v}\in V$ generic over $K$. 
    Then $\dim\psi(V) = \dim V- \dim V\cap(\mathbf{v}+T)$ and $\dim\psi'(V) = \dim V - \dim V\cap(\mathbf{v} +S_0+T)$. 
    On the other hand, since $\dim V\cap(\mathbf{v}+T)\leq\dim V\cap(\mathbf{v} +S_0+T)$, we get
    \begin{equation*}
        \dim V - \dim V\cap(\mathbf{v} +S_0+T) + \dim\psi(S) \leq \dim\psi(V) + \dim\psi(S) < n-\dim T,
    \end{equation*}
    which gives
    \begin{equation*}
        \dim V - \dim V\cap(\mathbf{v} +S_0+T) < n - \dim T -\dim\psi(S).
    \end{equation*}
    The fiber-dimension theorem and Corollary \ref{cor:perslikelycosets} give $\dim (S_0+T) = \dim \psi(S_0) + \dim T$ and $\dim\psi(S) = \psi(S_0)$, hence
    \[ \dim V - \dim V\cap(\mathbf{v} +S_0+T) < n - \dim (S_0+T).\]
    Therefore
    \begin{equation*}
        \dim V\cap(\mathbf{v} +S_0+T) > \dim V + (\dim S_0+T) - n,
    \end{equation*}
    which shows that the intersection between $V$ and $\mathbf{v} +S_0+T$ is atypical. 
    Let $X$ be a component of $V\cap(\mathbf{v}+ S_0+T)$ of maximal dimension. 
    By Theorem \ref{thm:weakzp} there is $H\in\mathscr{H}$ such that $X\subseteq V\cap(\mathbf{v}+H)$ and
    \begin{equation*}
        \dim V\cap(\mathbf{v}+ S_0+T) = \dim X \leq \dim V\cap (\mathbf{v}+H) + \dim( S_0+T)\cap H - \dim H.
    \end{equation*}
    Combining the last two inequalities gives
    \begin{equation*}
        \dim V\cap (\mathbf{v}+H) > \dim V + \dim (S_0+T) + \dim H - \dim (S_0+T)\cap H - n.
    \end{equation*}
    Since $\mathbf{v}\in V$ is generic over $K$ and $H$ is defined over $K$, 
    we still have $\dim\psi_H(V) = \dim V - \dim V\cap(\mathbf{v}+H)$, where $\psi_H:E_q^n\to E_q^n/H$ denotes the quotient map. 
    Furthermore, Lemma \ref{lem:dimcosetintersection} and the fiber-dimension theorem give that for any $\mathbf{b}\in S$ and any $\mathbf{a}\in S_0$,
    \[\dim S = \dim\psi_H(S) + \dim (\mathbf{c}+S_0)\cap(\mathbf{b}+H) =  \dim\psi_H(S) + \dim S_0\cap(\mathbf{a}+H).\]
    Thus, choosing $\mathbf{a}$ to be the identity gives
    \begin{align*}
        &\dim\psi_H(V) + \dim\psi_H(S)\\ &= \dim V-\dim V\cap(\mathbf{v}+H) + \dim S - \dim S_0\cap H\\
        &< \dim V - (\dim V + \dim (S_0+T) + \dim H - \dim (S_0+T)\cap H - n) + \dim S_0-\dim S_0\cap H\\
        &=n-\dim H + \dim(S_0+T)\cap H+\dim S_0\cap T -\dim S_0\cap H- \dim T\\
        &= n-\dim H +\dim (S_0+H)\cap T - \dim T\\
        &\leq n-\dim H.
    \end{align*}
    This shows that $\mathscr{H}$ is the desired finite collection. 
\end{proof}

The following result is a $p$-adic analog of \cite[Theorem 3.3]{eterovic-scanlon}. 

\begin{thm}
Let $V\subseteq E_q^n$ be an irreducible subvariety. 
If $S$ is a weakly special subvariety of $E_q^n$ such that the intersection of $V$ and $S$ is persistently likely, then for every p-adically dense subset $A\subseteq E_q^n$,
\begin{equation*}
    \bigcup_{\mathbf{g}\in A}V\cap (\mathbf{g}+S)
\end{equation*}
is $p$-adically dense in $V$.
\end{thm}
\begin{proof}
We first observe that if $S = E_q^n$, then the result is immediate, so from now on we assume that $S$ is a proper special subvariety of $E_q^n$.

The proof proceeds by induction on $n$. 
Suppose $n=1$, then either $V$ is a point or $V=E_q$. 
Since $S$ is a proper weakly special subvariety of $E_q$, it must be a point $\{\xi\}$.
In order for $V$ and $S$ to have a persistently likely intersection, we must have  $V=E_q$, in which case we are done. 

Now we treat the general case.
Let $T\subset E_q^n$ be an irreducible algebraic subgroup, and let $\boldsymbol{\xi}\in E_q^n$ be such that $\boldsymbol{\xi}+ T=S$. 
Let $L\subset\mathbb{G}_m^n$ be a proper subtorus such that $\phi_q(L) = T$. 
Let $\mathcal{F}\subset\mathbb{G}_m^n(\CC_p)$ denote a fundamental domain for $\phi_q$, so that the restriction of $\phi_q$ to $\mathcal{F}$ is definable in the 1-h-minimal structure on $\mathbb{C}_p$ from Corollary \ref{cor:countable.language.Cpexp}. 
Set 
\[k:=\dim V + \dim S - n.\]

Let $U$ be a non-empty, definable, relatively open subset of $V$ (with respect to the $p$-adic topology). 
Let $\mathscr{H}$ be the finite collection of proper subgroups of $E_q^n$ obtained by applying Theorem \ref{thm:weakzp} to (the constant family) $V$. 
By the $h$-minimal fiber-dimension theorem, we may shrink $U$ if necessary (while preserving definability, openness and dimension) so that for every $H\in\mathscr{H}$ and every $\mathbf{x}\in U$ we have
\begin{equation}
    \label{eq:fibercondition}
    \dim U = \dim\psi_H(U) + \dim U\cap(\mathbf{x}+H),
\end{equation}
where $\psi_H:E_q^n\to E_q^n/H$ is the quotient map. 

The relation:
\begin{equation*}
    \mathcal{R}:=\{(\mathbf{u},\mathbf{g})\in U\times\mathbb{G}_m^n(\CC_p) : \mathbf{u}\in\phi_q\left((\mathbf{g}L)\cap\mathcal{F}\right)\}
\end{equation*}
is definable. 
We will now compute the dimension of the definable set $\mathcal{R}$ in two different ways, first by considering the coordinate projection to $U$.

\begin{claim}
\label{claim:dimR}
$\dim \mathcal{R} = \dim \mathbb{G}_m^n(\CC_p) + k = n+k$.
\end{claim}
\begin{proof}
Given $\mathbf{u}\in U$, define $\mathcal{R}_\mathbf{u}:=\{\mathbf{g}\in\mathbb{G}_m^n : (\mathbf{u},\mathbf{g})\in\mathcal{R}\}$. 
Now choose $\mathbf{z}_0\in\mathcal{F}$ such that $\phi_q(\mathbf{z}_0)=\mathbf{u}$. 

Observe that we can write $\mathbf{z}_0 = \boldsymbol{1}\cdot \mathbf{z}_0$, and since $\boldsymbol{1}\in L$, we conclude that $\mathbf{z}_0\in\mathcal{R}_\mathbf{u}$, so in particular $\mathcal{R}_\mathbf{u}$ is non-empty. 
This also shows that the coordinate projection $\mathrm{pr}_U:\mathcal{R}\to U$ is surjective, so $\dim\mathrm{pr}_U(\mathcal{R})= \dim U$. 

Consider the map $\alpha:L\to\mathcal{R}_\mathbf{u}$ given by $\alpha(\mathbf{g}):=\mathbf{z}_0\mathbf{g}$. 
We will now check that $\alpha$ is a definable bijection between $L$ and $\mathcal{R}_\mathbf{u}$. 
\begin{enumerate}
    \item First, we check that $\alpha$ is well-defined. 
    Given $\mathbf{g}\in L$ we have $\mathbf{z}_0\mathbf{g}L = \mathbf{z}_0L$.
Since $\mathbf{z}_0\in\mathcal{R}_\mathbf{u}$, this shows that $\mathbf{z}_0\mathbf{g}\in\mathcal{R}_\mathbf{u}$.

\item Given $\mathbf{g}\in\mathcal{R}_\mathbf{u}$ there is $\mathbf{c}\in L$ such that $\phi_q(\mathbf{g}\mathbf{c})=\mathbf{u}$ and $\mathbf{c}\mathbf{g}\in\mathcal{F}$. 
Since $\phi_q(\mathbf{z}_0)=\mathbf{u}$, there is $\mathbf{m}\in\mathbb{Z}^n$ such that $\mathbf{c}\mathbf{g} = \mathbf{z}_0q^{\mathbf{m}}$, but since both  $\mathbf{z}_0$ and $\mathbf{c}\mathbf{g}$ are in $\mathcal{F}$, then $\mathbf{m}=\boldsymbol{0}$. 
Therefore, $\mathbf{g} = \mathbf{z}_0\mathbf{c}^{-1} = \alpha(\mathbf{c}^{-1})$.
Since $\mathbf{c}^{-1}\in L$, we conclude that $\alpha$ is surjective.

\item Lastly, it is immediate that $\alpha$ is injective and definable. 
\end{enumerate}
With these properties of $\phi$ verified, we immediately get $\dim \mathcal{R}_\mathbf{u} = \dim L$.

Now we prove the claim. 
By the $h$-minimal fiber-dimension theorem, there is $\mathbf{u}\in U$ such that $\dim\mathcal{R} = \dim\mathcal{R}_\mathbf{u} + \dim\mathrm{pr}_U(\mathcal{R})$. 
Then 
\begin{equation*}
    \dim\mathcal{R} = \dim\mathcal{R}_\mathbf{u} + \dim\mathrm{pr}_U(\mathcal{R})
    = \dim L  + \dim U
    = \dim S + \dim V
    = n+k.\qedhere
\end{equation*}
\end{proof}

Next, we compute the dimension of $\mathcal{R}$ using the coordinate projection to $\mathbb{G}_m^n(\CC_p)$. 
Given $\mathbf{g}\in\mathbb{G}_m^n(\CC_p)$, let $\mathcal{R}_{\mathbf{g}}:=\{\mathbf{u}\in U : (\mathbf{u}, \mathbf{g})\in\mathcal{R}\}$. 

Given $i\leq\dim V$, define
\begin{equation*}
    B_i:=\{\mathbf{g}\in\mathbb{G}_m^n:\dim\mathcal{R}_{\mathbf{g}}=i\}.
\end{equation*}
By Proposition \ref{prop:dim.intersection}, we know that the local dimension at each point of $U\cap \phi_q(\mathbf{g}L\cap\mathcal{F})$ is at least $\dim U+\dim L-n=k$, so $B_i=\emptyset$ whenever $i<k$.

\begin{claim}
\label{claim:dimB}
If $i>k$, then $B_i =\emptyset$.
\end{claim}
\begin{proof}
Suppose $i>k$ is such that $B_i\neq\emptyset$.  
Set $\mathbf{a} :=\phi_q(\mathbf{g})$. 
Since $U\cap \phi_q(\mathbf{g}L\cap\mathcal{F}) \subseteq V\cap(\mathbf{a}+S)$, there is an irreducible component $X$ of $V\cap(\mathbf{a}+S)$ of dimension at least $i$, hence
\[\dim X\geq\dim C = i > k = \dim V + \dim (\mathbf{a}+S) - n.\]
This shows that $X$ is an atypical component of $V\cap(\mathbf{a}+S)$. 
Let $H\in\mathscr{H}$ and $\mathbf{c}\in E_q^n$ be so that $X\subseteq \mathbf{c}+H$.
In fact, we may choose $\mathbf{c}\in C\subseteq U$. 

Let $\psi:E_q^n\to E_q^n/H$ denote the quotient map. 
Observe that by (\ref{eq:fibercondition})
\begin{equation*}
    \dim V =\dim U = \dim\psi(U) + \dim U\cap (\mathbf{c}+H).
\end{equation*}
Since we assume the intersection between $V$ and $S$ is persistently likely, Corollary \ref{cor:perslikelycosets} implies that $V$ and $\mathbf{a}+S$ are also persistently likely, so there exists $k'>0$ such that
\begin{equation*}
    \dim\psi(U) + \dim\psi\left(\mathbf{a}+S\right) = n-\dim H+k'.
\end{equation*}
Again using the fiber-dimension theorem, we get
\begin{equation*}
    \dim S =\dim (\mathbf{a}+S) = \dim\psi\left(\mathbf{a}+S\right) + \dim (\mathbf{a}+S)\cap(\mathbf{c}+H).
\end{equation*}
Putting these equalities together gives
\begin{align*}
    \dim U\cap (\mathbf{c}+H) + \dim (\mathbf{a}+S)\cap(\mathbf{c}+H) 
    &= \dim (\mathbf{a}+S) + \dim V - \dim\psi(U) - \dim\psi\left(\mathbf{a}+S\right)\\
    &= \dim V+ \dim S - (n-\dim H+k')\\
    &= \dim H + k-k'.
\end{align*}
Now we look at the intersection between $V\cap (\mathbf{c}+H)$ and $(\mathbf{a}+S)\cap(\mathbf{c}+H)$ inside $\mathbf{c}+H$.
Since $X$ is contained in both $V\cap (\mathbf{c}+H)$ and $(\mathbf{a}+S)\cap(\mathbf{c}+H)$, and $\dim X\geq i>k$, we can use the equalities above along with Theorem \ref{thm:weakzp} to get
\[i\leq \dim X \leq \dim V\cap(\mathbf{c}+H) + \dim S\cap H - \dim H = k-k'\]
which is a contradiction.
\end{proof}
Let $B$ be the image of the projection of $\mathcal{R}$ to $\mathbb{G}_m^n$.
By Claim \ref{claim:dimB} we know that
\[B = \bigcup_{i\in\mathbb{N}}B_i = B_{k}.\]
Combining this with Claim \ref{claim:dimR} and the fiber-dimension theorem we get
\[n + k = \dim\mathcal{R} = \dim B + k,\]
which gives $\dim B = n$. 
By cell decomposition, this implies that $B$ has non-empty interior (as a subset of $\mathbb{G}_m^n$) in the $p$-adic topology.
Since $A$ is $p$-adically dense in $E_q^n$ and $\phi_q$ is a continuous map, there exists $\mathbf{b}\in B$ such that $\boldsymbol{\xi}:=\phi_q(\mathbf{b})\in A$.
Since $\mathbf{b}\in B$, there must exist $\mathbf{u}\in U$ such that $(\mathbf{u},\mathbf{b})\in\mathcal{R}$, and by the definition of $\mathcal{R}$ we conclude that
\[\mathbf{u}\in U\cap(\mathbf{d} +S),\]
thus proving the theorem. 
\end{proof}

\bibliographystyle{alphaurl}
\bibliography{references}{}

\end{document}